\documentclass{llncs}
\usepackage{makeidx}  
\usepackage{amsmath,amssymb,stackrel}
\usepackage[all]{xy}
\usepackage{mathtools}
\usepackage{bm}
\DeclareMathOperator{\Cov}{Cov}
\DeclareMathOperator{\Derivative}{D}
\DeclareMathOperator{\E}{E} 

\DeclareMathOperator{\Supp}{Supp}

\DeclareMathOperator{\Phiexp}{\exp_{\phi}}
\DeclareMathOperator{\Philn}{\ln_{\phi}}
\DeclareMathOperator{\Qexp}{\exp_{q}}
\DeclareMathOperator{\Qln}{\ln_{q}}
\DeclareMathOperator{\Kexp}{\exp_{\kappa}}
\DeclareMathOperator{\Kln}{\ln_{\kappa}}
\DeclareMathOperator{\Range}{Range}

\newcommand{\rangeof}[1]{\Range\left(#1\right)}
\newcommand{\kexpof}[1]{\Kexp\left(#1\right)}
\newcommand{\klnof}[1]{\Kln\left(#1\right)}
\newcommand{\lnof}[1]{\ln\left(#1\right)}
\newcommand{\qexpof}[1]{\Qexp\left(#1\right)}
\newcommand{\qlnof}[1]{\Qln\left(#1\right)}
\newcommand{\phiexpof}[1]{\Phiexp\left(#1\right)}
\newcommand{\philnof}[1]{\Philn\left(#1\right)}
\newcommand{\Bspace}[1]{B_{#1}}

\newcommand{\absoluteval}[1]{\left\vert#1\right\vert}
\newcommand{\chartdom}{\mathcal U}
\newcommand{\covat}[3]{\Cov_{#1}\left(#2,#3\right)}
\newcommand{\densities}{\mathcal P_{\ge}}
\newcommand{\derivby}[1]{\frac{d}{d#1}}
\newcommand{\dirderat}[3]{\Derivative  {#1} \left(#2\right) #3}
\newcommand{\domainof}[1]{\mathcal D\left(#1\right)}

\newcommand{\emanifold}{\prescript{\text{e}}{}{\mathcal P}}
\newcommand{\etransport}[2]{\prescript{\text{e}}{} {\mathbb U} _ {#1} ^ {#2}}
\newcommand{\euler}{\mathrm{e}}
\newcommand{\expectat}[2]{{\E}_{#1}\left[#2\right]}
\newcommand{\expectof}[1]{\E\left(#1\right)}
\newcommand{\expof}[1]{\exp\left(#1\right)}

\newcommand{\logof}[1]{\log\left(#1\right)}
\newcommand{\mBspace}[1]{\prescript{*}{}{B}_{#1}}
\newcommand{\maxexp}[1]{{\mathcal E}\left(#1\right)}
\newcommand{\maxmix}[1]{\prescript{*}{}{\mathcal E}\left(#1\right)}

\newcommand{\mtransport}[2]{\prescript{\text{m}}{} {\mathbb U} _ {#1} ^ {#2}}
\newcommand{\nonnegreals}{\mathbb{R}_{\ge 0}}
\newcommand{\normat}[2]{\left\Vert#2\right\Vert_{#1}}

\newcommand{\pdensities}{\mathcal P_>}

\newcommand{\reals}{\mathbb{R}}
\newcommand{\scalarat}[3]{\left\langle#2,#3\right\rangle_{#1}}
\newcommand{\scalarof}[2]{\left\langle#1,#2\right\rangle}
\newcommand{\sdensities}{\mathcal P_{1}}
\newcommand{\setof}[2]{\left\{#1 : #2 \right\}}
\newcommand{\set}[1]{\left\{#1\right\}}

\newcommand{\supp}[1]{\Supp{#1}}
\newcommand{\transport}[2]{{\mathbb U} _ {#1} ^ {#2}}
\newcommand{\upgamma}[2]{\Gamma\left(#1,#2\right)}

\begin{document}
\frontmatter          
\pagestyle{headings}  

%
\mainmatter              
\title{Nonparametric Information Geometry}
\titlerunning{Nonparametric IG}  
%
\author{Giovanni Pistone}
\authorrunning{G. Pistone} 
%
\tocauthor{Giovanni Pistone}
\institute{Collegio Carlo Alberto, Via Real Collegio 30, 10024
  Moncalieri, Italy,\\
\email{giovanni.pistone@carloalberto.org},\\ 
WWW home page: \texttt{http://www.giannidiorestino.it/index.html}}

\maketitle              

\begin{abstract} The differential-geometric structure of the set of positive densities on a given measure space has raised the interest of many mathematicians after the discovery by C.R. Rao of the geometric meaning of the Fisher information. Most of the research is focused on parametric statistical models. In series of papers by author and coworkers a particular version of the nonparametric case has been discussed. It consists of a minimalistic structure modeled according the theory of exponential families: given a reference density other densities are represented by the centered log likelihood which is an element of an Orlicz space. This mappings give a system of charts of a Banach manifold. It has been observed that, while the construction is natural, the practical applicability is limited by the technical difficulty to deal with such a class of Banach spaces. It has been suggested recently to replace the exponential function with other functions with similar behavior but polynomial growth at infinity in order to obtain more tractable Banach spaces, e.g. Hilbert spaces. We give first a review of our theory with special emphasis on the specific issues of the infinite dimensional setting. In a second part we discuss two specific topics, differential equations and the metric connection. The position of this line of research with respect to other approaches is briefly discussed.
\keywords{Information Geometry, Banach Manifold}
\end{abstract}
%


\section{Introduction}
In the present paper we follow closely the presentation of Information Geometry developed by S.-I. Amari and coworkers, see e.g. in \cite{amari:82}, \cite{amari:85}, \cite{amari:87ims}, \cite{amari|nagaoka:2000}, with the specification that we want to construct a Banach manifold structure in the classical sense, see e.g. \cite{bourbaki:71variete} or \cite{lang:1995}, without any restriction to parametric models. We feel that the non parametric approach is of interest even in the case of a finite state space. We build upon our previous work in this field, namely \cite{pistone|sempi:95}, \cite{pistone|rogantin:99}, \cite{gibilisco|pistone:98}, \cite{cena:2002}, \cite{cena|pistone:2007}, \cite{malago|matteucci|dalseno:2008}, \cite{imparato:thesis}, \cite{brigo|pistone:2009arXiv:0901.1308}, \cite{malago|pistone:1012.0637}, \cite{pistone:2009EPJB}, \cite{pistone:2009SL}, \cite{imparato|trivellato:2010},\cite{malago|matteucci|pistone:2011a}, \cite{malago|matteucci|pistone:2011b}, \cite{malago:2012thesis},\cite{malago|matteucci|pistone:2013CEC}. Other contributions are referred to in the text below. We do not discuss here the non commutative/quantum case as developed e.g., in \cite{gibilisco|isola:1999}, \cite{jencova:2006} and the review in \cite{gibilisco|riccomagno|rogantin|wynn:2010}.

The rest of this introductory section contains a review of relevant facts related with the topology of Orlicz spaces which are the model spaces in our manifold structure. The review part is based on previous joint work with M. P. Rogantin \cite{pistone|rogantin:99} and A. Cena \cite{cena|pistone:2007}, but a number of examples and remarks are added in order to clarify potential issues and possible applications. The exponential manifold (originally introduced in the joint work with C. Sempi \cite{pistone|sempi:95}) is critically reviewed in Sec. \ref{sec:ExponentialManifold}, together with applications. Differential equations are discussed in Sec. \ref{sec:diff-eq}, with examples. Sec. \ref{sec:hilbertbundle} deals with the Hilbert bundle of the exponential manifold and the computation of the metric derivative. It builds upon previous work on non parametric connections with P. Gibilisco  \cite{gibilisco|pistone:98}. A variation on exponential manifolds is introduced in Sec. \ref{cha:deform-expon-manif} to show it could be developed along the lines previously discussed.  

\subsection{Model spaces}
In this paper we consider a fixed $\sigma$-finite measure space $(\Omega, \mathcal F, \mu)$ and we denote by $\pdensities$ the set of all densities which are positive $\mu$-a.s. The set of densities, without any further restriction is $\densities$, while $\sdensities$ is the set of measurable functions $f$ with $\int f\ d\mu = 1$. In the finite state space case, i.e. $\#\Omega < \infty$, $\sdensities$ is a plane, $\densities$ is the simplex, $\pdensities$ its topological interior. In the infinite case, the setting is much more difficult: we concentrate here mainly on strictly positive densities and we construct its geometry by taking as a guiding model the theory of exponential families, see \cite{efron:1975}, \cite{barndorff-nielsen:78}, \cite{brown:86}, \cite{letac:92}. A non parametric approach we use was initially suggested by P. Dawid \cite{dawid:75,dawid:1977AS}. A geometry derived from exponential families is intrinsically bases on the positivity of the densities, see \cite{gzyl|recht:2006geometryI,gzyl|recht:2006geometryII}.

At each $p \in \pdensities$ we associate a set of densities of the form $q = \euler^{u-K} \cdot p$, where $u$ belongs to a suitable Banach space $B_p$ and $K$ is a constant depending on $p$ and $u$. The mapping $u \mapsto q$ will be one-to-one and its inverse $s_p \colon q \mapsto u$ will be a chart of our exponential manifold $\emanifold = (\pdensities, \set{s_p})$. As we do not have manifold structures on the set of positive densities other that the exponential one, in the following the manifold and the set are both denoted by $\pdensities$.

We refer to \cite[\S 5-7]{bourbaki:71variete} and \cite{lang:1995} for the theory of manifolds modeled on Banach spaces. According to this definition, a manifold is a set $\pdensities$ together with a collection or atlas of charts $s \colon \chartdom \to B$ from a subset $\chartdom \subset \pdensities$ to a Banach space $B$ such that for each couple of charts the transition maps $s'\circ s^{-1} \colon s(\chartdom\cap\chartdom') \to s'(\chartdom\cap\chartdom')$ are smooth functions from an open set of $B$ into an open set of $B'$. In this geometric approach, $\pdensities$ is a set, while all structure is in model spaces $B$.

It should be noted that that the Banach spaces are not required to be equal, as the finite dimensional case seems to suggest, but they should be isomorphic when connected by a chart. Actually this freedom is of much use in our application to statistical model, but requires a careful discussion of the isomorphism. Precisely, at each $p \in \pdensities$, the model space $B_p$ for our manifold is an Orlicz space of centered random variables, see \cite{krasnoselskii|rutickii:61}, \cite[Chapter II]{musielak:1983}, \cite{rao|ren:2002}, \cite[Ch 8]{adams|fournier:2003}. We review briefly our notations and recall some basic facts from these references. 

If both $\phi$ and $\phi^{-1}=\phi_*$ are monotone, continuous functions on $\nonnegreals$ onto itself, we call the pair
\begin{equation*}
  \Phi(x) = \int_0^{\absoluteval x} \phi(u) \ du, \quad \Phi_*(y) = \int_0^{\absoluteval y} \phi^{-1}(v) \ dv,
\end{equation*}
a \emph{Young pair}. Each Young pair satisfies the Young inequality $\absoluteval{xy} \le \Phi(x) + \Phi_*(y)$, with equality if, and only if, $y = \phi(x)$. The relation in a Young pair is symmetric and either element is called a Young function. 
\begin{example}[Young pairs]\label{ex:Youngpairs}
We will use the following Young pairs:
\begin{equation*}
  \begin{array}{l|c|c|c|c}
   & \phi & \phi^{-1}=\phi_* & \Phi & \Phi_* \\
\hline
\text{(a)} & \logof{1+u} & \euler^v - 1 & (1+\absoluteval x)\logof{1+\absoluteval x} - \absoluteval x & \euler^{\absoluteval y} - 1 - \absoluteval y \\
\text{(b)} & \sinh^{-1} u & \sinh v & \absoluteval x \sinh^{-1}\absoluteval x - \sqrt{1+x^2} + 1 & \cosh y - 1 \\ 
(c) & \log^+ u & \euler^v & \absoluteval x \log^+ \absoluteval x - (x -1)^+ & \euler^{\absoluteval y} -1 \\
(2) & u & v & \frac 12 x^2 & \frac12 y^2
  \end{array}
\end{equation*}
As $\logof{1+u} \le \sinh^{-1} u = \logof{u + \sqrt{1+u^2}} \le a \logof{1+u}$ if $u \ge 0$ and $a > 1$, the pairs (a) and (b) are equivalent i.e. $\Phi_{\text{a}} \le \Phi_{\text{b}} \le a\Phi_{\text{a}}$ if $a >1$. Moreover, from $\Phi_{\text{a}}(x) = \int_0^{x} (x-u)/(1+u)\ du$ if $x \ge 0$, we obtain an instance of the so called $\Delta_2$-condition,  $\Phi_{\text{a}}(ax) \le a^2 \Phi_{\text{a}}(x)$. This condition is not satisfied by $\Phi_{\text{a}*}$ as $\Phi_{\text{a}*}(2y)/\Phi_{\text{a}*}(y)$ is unbounded as $y \to \infty$. The listed pairs satisfy $\Phi_{\text{(a)}}(x) \le \Phi_{\text{(2)}} \le \Phi_{\text{(a)}*}$.
\end{example}

 In fact, around each $p$ we consider densities of the form $q \propto \euler^v$ for some random variable $v$ and, moreover, we require the one dimensional exponential family $q(t) \propto \euler^{tv}$ be defined for each $t$ in an open interval $I$ containing 0. In other words, we require the moment generating function  $t \mapsto \int \euler^{iv}p\ d\mu = \expectat p {\euler^{tv}}$ to be finite in a neighbourhood of 0. The set of such random variables $v$ is a vector space and a Banach space for a properly defined norm. We discuss below those pars of the theory which are relevant for the definition of exponential manifold.

If $\Phi(x) = \cosh x - 1$, a real random variable $u$ belongs to the vector space $L^{\Phi}(p)$ if  $\expectat p {\Phi(\alpha v)} < +\infty$ for some $\alpha > 0$. A norm is obtained by defining the set $\setof{v}{\expectat p {\Phi(\alpha v)} \le 1}$ to be the closed unit ball. It follows that the open unit ball consists of those $u$'s such that $\alpha u$ is in the closed unit ball for some $\alpha > 1$. The corresponding norm $\Vert \cdot \Vert_{\Phi, p}$ is called Luxemburg norm and defines a Banach space, see e.g. \cite[Th 7.7]{musielak:1983}. The function $\cosh -1$ has been chosen here because the condition $\expectat p {\Phi(\alpha v)} < +\infty$ is clearly equivalent to $\expectat p {\euler^{t v}} < +\infty$ for $ t \in [-\alpha,\alpha]$, but other choices will define the same Banach space e.g., $\Phi(x) = \euler^{|x|}-|x|-1$. By abuse of notation, we will denote all these equivalent functions by $\Phi$.

The main technical issue in working with Orlicz spaces such as $L^{(\cosh -1)}(p)$ is the regularity of its unit sphere $S = \setof{u}{\normat {(\cosh-1),p} u = 1}$. In fact, while $\expectat p {\cosh u - 1} = 1$ implies $u \in S$, the latter implies $\expectat p {\cosh u - 1} \le 1$. Subspaces of $L^{\Phi}$ where this cannot happen are called \emph{steep}. If the state space is finite, the full space is steep, see the Ex. \ref{ex:boolean} and \ref{ex:nonsteep} below. The relevance of steep families in exponential families is discussed in \cite{barndorff-nielsen:78}. Steepness is important when related with the idea of embedding. Consider the mapping $\Phi_+^{-1} \colon \pdensities \ni p \mapsto v = \Phi_+^{-1}(p)$, $\Phi_+=\Phi_{|\reals_>}$. Then $\int \Phi(v)\ d\mu = \int p \ d\mu = 1$ hence $\normat{\Phi} u = 1$ and we have an embedding of $\pdensities$ into the sphere of a Banach space.

\begin{example}[Boolean state space]\label{ex:boolean}
In the case of a finite state space, the moment generating function is finite everywhere, but its computation can be challenging. We discuss in particular the boolean case $\Omega=\set{+1,-1}^n$ with counting reference measure $\mu$ and uniform density $p(x)=2^{-n}$, $x\in \Omega$. In this case there is a huge literature from statistical physics, e.g., \cite[Ch. VII]{gallavotti:1999short}. A generic real function on $\Omega$---called in the machine learning literature pseudo-boolean \cite{boros|hammer:2002}---has the form $u(x)=\sum_{\alpha \in L} \hat u(\alpha) x^\alpha$, with $L = \set{0,1}^n$, $x^\alpha = \prod_{i=1}^n x_i^{\alpha_i}$, $\hat u(\alpha) = 2^{-n} \sum_{x \in \Omega} u(x) x^\alpha$. 

As $\euler^{a x} = \cosh(a) + \sinh(a) x$ if $x^2 = 1$ i.e. $x = \pm 1$, we have
 \begin{align*}
  \euler^{tu(x)} &= \expof{\sum_{\alpha \in \supp{\hat{u}}} t\hat{u}(\alpha) x^\alpha} \\
&= \prod_{\alpha \in \supp{\hat{u}}} \left(\cosh(t\hat{u}(\alpha)) + \sinh(t\hat{u}(\alpha)) x^\alpha \right) \\&= \sum_{B \subset \supp{\hat u}}\prod_{\alpha \in B^c} \cosh(t\hat{u}(\alpha)) \prod_{\alpha \in B} \sinh(t\hat{u}(\alpha)) x^{\sum_{\alpha \in B}\alpha}.
\end{align*}

The moment generating function of $u$ under the uniform density $p$ is
\begin{equation*}
  t \mapsto \sum_{B \in \mathcal B(\hat u)} \prod_{\alpha \in B^c} \cosh(t\hat u(\alpha)) \prod_{\alpha \in B} \sinh(t\hat u(\alpha)),
\end{equation*}
where $\mathcal B(\hat u)$ are those $B \subset \supp{\hat u}$ such that $\sum_{\alpha \in B} \alpha = 0 \mod 2$. We have
\begin{equation*}
  \expectat p \Phi(tu) =  \sum_{B \in \mathcal B_0(\hat u)} \prod_{\alpha \in B^c} \cosh(t\hat u(\alpha)) \prod_{\alpha \in B} \sinh(t\hat u(\alpha)) - 1,
\end{equation*}
where $\mathcal B_0(\hat u)$ are those $B \subset \supp{\hat u}$ such that $\sum_{\alpha \in B} \alpha = 0 \mod 2$ and moreover $\sum_{\alpha \in \supp{\hat u}} \alpha = 0$. 

If $S$ is the $\set{1,\dots,n}\times\supp{\hat u}$ matrix with elements $\alpha_i$ we want to solve the system $Sb = 0 \mod 2$ to find all elements of $\mathcal B$; we want to add the equation $\sum b = 0 \mod 2$ to find $\mathcal B_0$. The simplest example is the simple effect model $u(x) = \sum_{i=1}^n c_i x_i$.
\end{example}
\begin{example}[The sphere $S$ is not smooth]\label{ex:nonsteep}
We look for the moment generating function of the density $p(x) \propto (a+x)^{-\frac32} \euler^{-x}$, $x >0$, where $a > 0$. From the incomplete gamma integral
\begin{equation*}
\upgamma{-\frac12}x = \int_x^\infty s^{-\frac12-1} \euler^{-s}\ ds, \quad x > 0,
\end{equation*}
we have for $\theta, a > 0$,
\begin{equation*}
\derivby x \upgamma{-\frac12}{\theta(a+x)} = -\theta^{-\frac12}\euler^{-\theta a}(a+x)^{-\frac32} \euler^{-\theta x}.
\end{equation*}

We have, for $\theta \in \reals$ and $a > 0$
\begin{equation}\label{eq:upgamma-int}
 C(\theta,a) = \int_0^\infty (a+x)^{-\frac32} \euler^{-\theta x}\ dx =
 \begin{cases}
   \sqrt\theta \euler^{\theta a} \upgamma{-\frac12}{\theta a} & \text{if $\theta > 0$,} \\
\frac2{\sqrt a} & \text{if $\theta = 0$}, \\ +\infty & \text{if $\theta < 0$,} 
 \end{cases}
\end{equation}
or, using the Gamma distribution with shape $1/2$ and scale 1, $\upgamma{-\frac12}{x} = 2 x^{-1/2} \euler^{-x} - \sqrt \pi \left(1 - \Gamma(x; 1/2, 1)\right)$,
\begin{equation*}
 C(\theta,a) =
 \begin{cases}
   2 a^{-\frac12} - 2 \sqrt {\pi \theta} \euler^{\theta a} \left(1 - \Gamma(\theta a; 1/2, 1)\right) & \text{if $\theta \ge 0$,} \\ +\infty & \text{it $\theta < 0$.}  
 \end{cases}
\end{equation*}

The density $p$ is obtained from \eqref{eq:upgamma-int} with $\theta=1$,
\begin{equation*}
  p(x) = C(1,a)^{-1} (a+x)^{-\frac32} \euler^{-x} = \frac{(a+x)^{-\frac32} \euler^{-x}}{\euler^{a} \upgamma{-\frac12}{a}}, \quad x > o,
\end{equation*}
and, for the random variable $u(x) = x$, the function 
\begin{align}
 \alpha \mapsto \expectat p {\Phi(\alpha u)} &= \frac1{\euler^{a} \upgamma{-\frac12}{a}}\int_0^\infty  (a+x)^{-\frac32}\frac{\euler^{-(1-\alpha)x}+\euler^{-(1+\alpha)x}}2\ dx - 1 \notag \\
&= \frac{C(1-\alpha,a)+C(1+\alpha,a)}{2C(1,a)} - 1 \label{eq:nonsteep}
\end{align}
is convex lower semi-continuous on $\alpha \in \reals$, finite for $\alpha \in [-1,1]$, infinite otherwise, hence not steep. Its value at $\alpha = \pm1$ is
\begin{equation*}
  \expectat p {\Phi(u)} = \frac{C(0,a)+C(2,a)}{2C(1,a)} - 1 \quad \text{e.g., $=0.8037381$ if $a = \frac12$.}
\end{equation*}
\end{example}

If the functions $\Phi$ and $\Phi_*$ are Young pair, for each $u \in L^{\Phi}(p)$ and $v \in L^{\Phi_*}(p)$, such that $\normat {\Phi,p} u, \normat {\Phi_*,p} v \le 1$, we have $\expectat p {uv} \le 2$, hence
\begin{equation*}
  L^{\Phi_*}(p) \times L^{\Phi}(p) \ni (v,u) \mapsto \expectat p {uv}
\end{equation*}
is a duality mapping, $\absoluteval{\scalarat p u v} \le 2 \normat {\Phi_*,p} u \normat {\Phi,p} v$ . 

A sequence $u_n$, $n=1,2,\dots$ is convergent to 0 for such a norm if and only if for all $\epsilon > 0$ there exists  a $n(\epsilon)$ such that $n > n(\epsilon)$ implies $\expectat p {\Phi_1(\frac{u_n}{\epsilon})} \le 1$. Note that $|u|\le |v|$ implies
\begin{equation*}
  \expectat p {\Phi_1\left(\frac{u}{\Vert v \Vert_{\Phi_1, p}}\right)} \le   \expectat p {\Phi_1\left(\frac{v}{\Vert v \Vert_{\Phi_1, p}}\right)} \le 1
\end{equation*}
so that $\normat {\Phi_1, p} u \le \normat {\Phi_1, p} v$. 

In defining our manifold, we need to show that Orlicz spaces defined at different points of statistical models are isomorphic, we will use frequently the fact that following lemma, see \cite[Lemma 1]{cena|pistone:2007}.
\begin{lemma}\label{lemma:eq-norm}
 Let $p\in\mathcal M$ and let $\Phi_0$ be a Young function. If the Orlicz spaces $L^{\Phi_0}(p)$ and $L^{\Phi_0}(q)$ are equal as sets, then their norms are equivalent. 
\end{lemma}

The condition $u \in L^{\cosh-1}(p)$ is equivalent to the existence of the moment generating function $g(t) = \expectat p {\euler^{tu}}$ on a neighbourhoods of 0. The case when such a moment generating function is defined on all of the real line is special and defines a notable subspace of the Orlicz space see e.g., \cite{rao|ren:2002}. Such spaces could be the model of an alternative definition of as in \cite{grasselli:2001}.

In fact, the Banach space $L^{\Phi}(p)$, $\phi=\cosh -1$ is not separable, unless the basic space has a finite number of atoms. In this sense it is an unusual choice from the point of view of functional analysis and manifold's theory. However, $L^{\Phi}(p)$ is natural for statistics because for each $u \in L^{\Phi_1}(p)$ the Laplace transform of $u$ is well defined at 0, then the one-dimensional exponential model $p(\theta) \propto \euler^{\theta u}$ is well defined.  

However, the space $L^{\Phi_*}(p)$ is separable and its dual space is $L^{\Phi}(p)$, the duality pairing being $(u,v) \mapsto \expectat p {uv}$. This duality extends to a continuous chain of spaces:
\begin{equation*}
  L^{\Phi_1}(p) \to L^a(p) \to L^b(p) \to L^{\Psi_1}(p), \quad 1 < b \le 2, \quad \frac1a+\frac1b=1
\end{equation*}
where $\to$ denotes continuous injection.

From the duality pairing of conjugate Orlicz spaces and the characterization of the closed unit ball it follows a definition of dual norm on $L^{\Phi_*}(p)$:
\begin{equation*}
  N_p(v) = \sup\setof{\expectat p {uv}}{\expectat p {\Phi(u)} \le 1}.
\end{equation*}
\subsection{Moment generating functional and cumulant generating functional}
In this section we review a number of key technical results. Most of the results are related with the smoothness of the superposition operator $L^\Phi(p) \colon v \mapsto \exp \circ v$. Superposition operators on Orlicz spaces are discussed e.g. in \cite{krasnoselskii|rutickii:61} and \cite[Ch 4]{appell|zabrejko:1990}. Banach analytic functions are discussed in \cite{bourbaki:71variete}, \cite{upmeier:1985} and \cite{ambrosetti|prodi:1993}.

Let $p \in \mathcal \pdensities$ be given. The following theorem has been proved in \cite[Ch 2]{cena:2002}, see also \cite{cena|pistone:2007}.
\begin{proposition}\label{prop:expisanalytic}\ 
\begin{enumerate}
\item
For $a \geq 1$, $n = 0, 1, \dots$ and $u \in L^\Phi(p)$,
\begin{equation*}
 \lambda_{a,n}(u) \colon \left(w_1, \dots, w_n \right)  \mapsto  \dfrac{w_1}{a} \cdots \dfrac{w_n}{a}\ \euler^{\frac ua}
\end{equation*}
is a continuous, symmetric, $n$-multi-linear map from $L^\Phi(p)$ to $L^{a}\left(p\right)$.%
\item
$v \mapsto \sum_{n = 0}^\infty \frac{1}{n!} \left(\dfrac{v}{a}\right)^n$ is a power series from $L^{\Phi}(p)$ to $L^a(p)$ with radius of convergence $\geq 1$.
\item The superposition mapping $ v \mapsto \euler^{v/a}$ is an analytic function from the open unit ball of $L^\Phi(p)$ to $L^a(p)$.
\end{enumerate}
\end{proposition}

\begin{definition}
Let $\Phi = \cosh -1$ and $B_p = L^\Phi_0(p)$, $p \in \pdensities$. The \emph{moment generating functional} is $M_p  \colon L^{\Phi}(p) \ni u \mapsto \expectat p {\euler^u} \in \reals_> \cup \set {+\infty}$. The \emph{cumulant generating functional} is $K_p  \colon B_p \ni u \mapsto \log M_p(u) \in \reals_> \cup \set{+\infty}$.
\end{definition}
\begin{proposition} \label{thm:moment_analytic}\ 
\begin{enumerate}
\item $M_p (0) = 1$; otherwise, for each centered random variable $u \neq 0$, $M_p (u) > 1$.
\item $M_p$ is convex and lower semi-continuous, and its proper domain is a convex set which contains the open unit ball of $L^{\Phi}(p  )$; in particular the interior of such a domain is a non empty convex set.
\item $M_p$ is infinitely G\^ateaux-differentiable in the interior of its proper domain, the nth-derivative at $u$ in the direction $v \in L^{\Phi}(p)$ being
\begin{equation*}
  \left. \frac{d^n}{dt^n} M_p (u + tv) \right|_{t=0} = \expectat p {v^n \euler^u};
\end{equation*}
\item $M_p$ is bounded, infinitely Fr\'echet-differentiable and analytic on the open unit ball 
of $L^{\Phi}(p  )$, the nth-derivative at $u$ evaluated in $(v_1, \dots, v_n) \in L^{\Phi}(p  ) \times \cdots \times L^{\Phi}(p  )$ is 
\begin{equation*}
  D^n M_p (u) (v_1, \dots, v_n) = \expectat p {v_1 \cdots v_n \euler^u}.
\end{equation*}
\end{enumerate}
\end{proposition}
\begin{proposition} \label{cor:cumulant}\ 
\begin{enumerate}
\item $K_p (0) = 0$; otherwise, for each $u \neq 0$, $K_p (u) > 0$.
\item $K_p$ is convex and lower semi-continuous, and its proper domain is a convex set which contains the open unit ball of $\Bspace p$; in particular the interior of such a domain is a non empty convex set.
\item $K_p$ is infinitely G\^ateaux-differentiable in the interior of its proper domain.
\item $K_p$ is bounded, infinitely Fr\'echet-differentiable and analytic on the open unit ball 
of $\mathcal V_p$.
\end{enumerate}
\end{proposition}
Other properties of the key functional $K_p$ are described below as they relate directly to the exponential manifold.
\subsection{Families of Orlicz spaces}
In statistical models, we associate to each density $p$ a space of $p$-centered random variables to represent scores or estimating functions. For example, if the one-parameter statistical model $p(t)$, $t \in I$, $I$ open interval, is regular enough, then $u(t) = \frac{d}{dt} \log p(t)$ satisfies $\expectat{p(t)}{u(t)} = 0$ for all $t \in I$. It is crucial to discuss how the relevant spaces of $p$-centered random variables depend on the variation of the density $p$, that is it is crucial to understand the variation of the spaces $B_p = L_0^{\Phi}(p)$ and $\mBspace  p = L^{\Phi_*}_0(p)$ along a one-dimensional statistical model $p(t)$, $t\in I$. In Information Geometry, those spaces contain models for the tangent and cotangent spaces of the statistical models. On two different points of a regular model, they must be isomorphic, or, in particular, equal.

We use a peculiar notion of connection by arcs, which is different from what is usually meant with this name. Given $p, q \in \pdensities$, the exponential model $p\left(\theta\right) \propto p^{1-\theta}q^\theta$, $0 \le \theta \le 1$ connects the two given densities as end points of a curve, $p(\theta)\propto \exp\left(\theta \log \frac qp \right) \cdot p$, where $\log \frac qp$ is not in the exponential Orlicz space at $p$ unless $\theta$ can be extended to assume negative values. 
\begin{definition} \label{def:open-connected} We say that $p, q \in
  \pdensities$ are connected by an open exponential arc if there
  exist $r \in \pdensities$ and an open interval $I$, such that $p\left( t \right)
  \propto \euler^{tu} r$, $t \in I$, is an exponential model containing
  both $p$ and $q$ at $t_0, t_1$ respectively. By the change of parameter $s = t - t_0$, we can always reduce to the case where $r=p$ and $u \in L^{\Phi}(p)$.
\end{definition}

The open connection of Def. \ref{def:open-connected} is an  equivalence relation.

\begin{definition}\label{def:maximalexpmod}
Let us denote by $\mathcal S_p$ the interior of the proper domain of the cumulant generating functional $K_p$. For every density $p \in \pdensities$, the \emph{maximal exponential model at $p$} is defined to be the family of densities
\begin{equation*}
  \mathcal{E}\left(p\right) := 
  \setof{\euler^{u - K_p \left(u\right)}\cdot p}{ u \in \mathcal S_p}.
\end{equation*}
\end{definition}
\begin{proposition} \label{th:22} 
  The following statements are equivalent:
  \begin{enumerate}
  \item \label{th:22-2} $q \in \mathcal M$ is connected to $p$ by an open exponential arc;
  \item \label{th:22-1} $q \in \mathcal E(p)$;
  \item \label{th:22-4} $\mathcal E(p) = \mathcal E(q)$;
  \item \label{th:22-3} $\log \frac q p$ belongs to both $L^{\Phi_1}(p)$ and $L^{\Phi_1}(q)$.
  \item \label{th:22-5} $L^{\Phi_1}(p)$ and $L^{\Phi_1}(q)$ are equal as vector spaces and their norms are equivalent.
\end{enumerate}
\end{proposition}
In the following proposition we have collected a number of properties of the maximal exponential model $\mathcal E(p)$ which are relevant for its manifold structure.
\begin{proposition}\label{pr:misc}\ 
  Assume $q = \euler^{u - K_p(u)} \cdot p \in \maxexp p$.
  \begin{enumerate}
\item \label{item:firsttwo} The first two derivatives of $K_p$ on $\mathcal S_p$ are
\begin{align*}
D K_p(u) v &= \expectat q v,  \\
D^2 K_p(u)(v_1, v_2) &= \covat q {v_1}{v_2}
\end{align*}
\item The random variable $\frac q p -1$ belongs to $\mBspace p$ and
\begin{equation*}
  D {K_p(u)} v = \expectat p {\left( \frac q p - 1\right) v}.
\end{equation*}
In other words the gradient of $K_p$ at $u$ is identified with an element of $\mBspace p$, denoted by $\nabla K_p(u) = e^{u - K_p(u)} - 1=\frac q p -1$.
\item
The mapping $B_p \ni u\mapsto \nabla K_p(u) \in \mBspace p$ is monotonic, in particular one-to-one.
\item \label{item:weakderiv}
The weak derivative of the map $\mathcal S_p \ni u\mapsto \nabla K_p(u)
  \in \mBspace p$ at $u$ applied to $w \in B_p$ is given by
\begin{equation*}
  D (\nabla K_p(u)){w}= \frac q p \left( w- \expectat q w
  \right),
\end{equation*}
and it is one-to-one at each point.

\item \label{pr:misc-upq} The mapping $\mtransport pq : v \mapsto \frac pq \, v$ is an isomorphism of $\mBspace p$ onto $\mBspace q$.
\item \label{pr:misc-1} $q/p \in L^{\Phi_*}(p)$.
  \item \label{pr:misc-2} $D \left(q \Vert p \right) = D K_p(u)  u - K_p(u)$ with $q=\euler^{u-K_p(u)}p$,  in particular $- D(q \, \| \,  p) < +\infty$. 
  \item  \label{pr:misc-3} 
    \begin{equation*}
    B_q = L_0^{\Phi_1}(q) = \setof{u \in L^{\Phi_1}(p)}{ \expectat p {u \frac qp}  = 0}.  
    \end{equation*}
  \item \label{pr:misc-5} $\etransport pq \colon u \mapsto u - \expectat q u$ is an isomorphism  of $B_p$ onto $B_q$.
  \end{enumerate}
\end{proposition}
\section{Exponential and mixture manifolds}\label{sec:ExponentialManifold}

\subsection{Exponential manifold}
\label{sec:exponential-manifold}
If $p,q \in \mathcal M$ are connected  by an open exponential arc, then the random variable  $u \in \mathcal S_p$ such that  $q \propto \euler^{u} p $ is unique and it is equal to $ \log \frac qp - \expectat p {  \log \frac qp }$. 
In fact, $q \propto \euler^{u} p$ for some $u \in L^{\Phi_1} \left( p \right)$ if and only if $u - \log \frac qp$ is a constant.
If $u \in \mathcal S_p \subset B_p$, then $u - \log \frac qp = K_p \left(u\right)$ 
and, as $u$ is centered, it follows that $- \expectat p {\log \frac qp }= K_p\left(u\right)$ and $u =  \log \frac qp - \expectat p {  \log \frac qp }$. Indeed,   $u$ is the projection of $\log \frac qp$ onto $B_p$ in the split $L^{\Phi_1}\left(p\right) = B_p \oplus \left< 1 \right>$.
\begin{definition}\label{def:echarts}
  We define two one-to-one mappings: the \emph{parameterization} or \emph{patch} $e_p \colon \mathcal S_p \to \maxexp p$, $e_p(u) = \euler^{u-K_p(u)}\cdot p$ and the \emph{chart} $s_p \colon \maxexp p \to \mathcal S_p$, $s_p(q) = \logof{frac{q}{p}} - \expectat p { \logof{\frac{q}{p}} }$.
\end{definition}
\begin{proposition}\label{prop:expisaffine}
If $p_1, p_2 \in \mathcal E\left(p\right)$, then the transition mapping  $s_{p_2} \circ e_{p_1} \colon \mathcal S_{p_1} \to \mathcal S_{p_2}$ is the restriction of an affine function from $\Bspace {p_1} \to \Bspace {p_2}$
\begin{equation*}
 u  \mapsto   u + \logof{\frac{p_1}{p_2}} - \expectat {p_2} { u + \logof{\frac{p_1}{p_2}}}.
\end{equation*}
\end{proposition}

The derivative of the transition map $s_{p_2} \circ e_{p_1} $ is the  isomorphism of $B_{p_1}$ onto $B_{p_2}$
\begin{equation*}
  B_{p_1} \ni u \mapsto u - \expectat {p_2}{u} = \etransport{p_1}{p_2} \in B_{p_2}. 
\end{equation*}

\begin{definition}The \emph{exponential manifold} is defined by the atlas of charts in Def. \ref{def:echarts}. It is an \emph{affine manifold} because of Prop. \ref{prop:expisaffine}. Each $\maxexp p$ is a connected component. 
\end{definition}

A metric topology called e-topology is induced by the exponential manifold on $\pdensities$, namely a sequence  $\left\{ p_n \right\}$, ${n \in \mathbb{N}}$, is e-convergent to $p$ if and only if 
  sequences $\left\{ p_n / p \right\}$ and $\left\{ p / p_n \right\}$ are convergent to $1$ in each 
  $L^\alpha \left( p \right)$, $\alpha>1$.

Mixture arcs are regular in each connected component $\mathcal E$ of the exponential manifold.
\begin{proposition}\label{pr:mix}\ 
\begin{enumerate}
\item 
If $q \in \mathcal E\left( p \right)$, then the mixture model 
$p \left( \lambda \right) = \left( 1 - \lambda \right) p + \lambda q \in \mathcal E\left(p\right)$ for $\lambda \in \left[ 0, 1 \right]$.
\item
An open mixture arc $p\left( t \right) =  \left(1 - t \right) p + t q$, $t \in \left] -\alpha, 1 + \beta \right[$, $\alpha,\beta > 0$ is e-continuous.
\end{enumerate}
\end{proposition}

\subsection{Mixture manifold}\label{sec:mix}

We are not able to define a mixture manifold with the same support as the exponential manifold. 
For each $p\in \pdensities$ and each $u \in \mathcal S_p$, $q = \euler^{u -K_p\left(u\right)}\cdot p$,  the derivative of $K_p $ at $u$,  in the direction $v \in B_p$, is $D K_p\left(u\right) \cdot v = \expectat p { \left(\frac{q}{p} -1 \right) v   }$,  and it is identified to its gradient $\nabla K_p(u) = q/p - 1 \in {\mBspace p}$. The mapping  $q \mapsto q/p -1 \in {\mBspace p}$ cannot be a chart because its values are bounded below by $-1$ but it is strongly reminiscent of the mean parameterization $\bm \eta = \nabla \psi(\bm \theta)$ in parametric exponential families $p_\theta = \expof{\bm{\theta \cdot T} - \psi(\bm\theta)} \cdot p_0$.

We move to the larger set $\sdensities = \setof{f}{\int f\ d\mu = 1} \supset \pdensities$ and for each $p \in \maxexp p$ we introduce  the subset $^*\mathcal{U}_p$ defined by the condition $\frac{q}{p} \in L^{\Phi_*(p)}$. Our chart is the map
\begin{equation}
  \label{eq:etamap}
  \eta_p \colon ^*\mathcal{U}_p \ni q \mapsto \frac{q}{p} - 1 \in \mBspace p.
\end{equation}

As $\eta_p(q)$, for $q \in \mathcal U_p \subset \mathcal E(p)$, equals $v \mapsto \expectat q v$, it is the non parametric version of the so called expectation parameter. This mapping is bijective and its inverse is:
\begin{equation*}
\eta_p^{-1} : {\mBspace p} \ni u \mapsto \left(u + 1 \right) p \in {^* \mathcal{U}_p}.
\end{equation*}
The collection of sets $\left\{ {^*\mathcal{U}_p} \right\}_{p \in \densities}$ 
is a covering of $\sdensities$.

There is a nice characterization of the elements of $^*\mathcal{U}_p \bigcap \densities$: they are all  the probability densities  with finite divergence with respect to $p$, see \cite[Prop 31]{cena|pistone:2007}. Moreover $\mathcal{U}_p \subset {^*\mathcal{U}_p}$ and  $p_1,p_2 \in \mathcal{E} \left( p \right)$ implies $^*\mathcal{U}_{p_1} = ^*\mathcal{U}_{p_2}$. In conclusion, we can define the mixture manifold as follows.

For each pair $p_1,p_2 \in \mathcal{E}\left(p\right)$ we have the affine transition map
\begin{equation*}
  \eta_{p_2} \circ \eta^{-1}_{p_1} : \left\{
    \begin{array}{rcl}
      ^*B_{p_1} & \rightarrow &  {^*B_{p_2}} \smallskip\\
      u & \mapsto &  u \, \dfrac{p_1}{p_2} + \dfrac{p_1}{p_2} -1
    \end{array}\right.
\end{equation*}
and the subset of $\sdensities$, $^*\mathcal{E}\left(p\right) = \setof {q \in \sdensities }{ \frac{q}{p} \in L^{\Phi_*} \left(p\right)}$, which is equal to $^*\mathcal{U}_q$ if $q \in \maxexp p$.
\begin{proposition}\label{th:mix}
  Let $p \in \pdensities$ be given. 
  The collection of charts 
  \begin{equation*}
    \setof { \left( {^*\mathcal{U}_{q}} , \eta_{q}\right) }{ q \in {\mathcal{E}} \left(p\right) }
  \end{equation*}
  is an affine $C^\infty$-atlas on ${^*\mathcal{E} \left(p\right)}$.
\end{proposition}

The \emph{mixture manifold} is defined by the atlas in Prop. \ref{th:mix}. The mixture manifold is an extension of the exponential manifold.

\begin{proposition}
For each density $p \in \pdensities$, the inclusion $\maxexp p \to \maxmix p$ is of class $C^\infty$.
\end{proposition}

 \subsection{Examples of applications}
\label{sec:exampl-appl}
\begin{example}[Divergence] The divergence $D(q\Vert r) = \expectat q {\logof{\frac qr}}$ is $C^\infty$ jointly in both variables for $q,r \in \maxexp p$. In fact, in the $p$ chart,  $u = s_p(q)$, $v = s_p(r)$ gives 
  \begin{align*}
    \mathcal S_p \times \mathcal S_p \ni (u,v) &\mapsto \expectat q {u - K_p(u) - v + K_p(v)} \\
&= K_p(v) - K_p(u) - D K_p(u) (v-u). 
  \end{align*}
In the exponential chart the KL divergence is the Bregman divergence of $K_p$.

The partial derivative in $u$ in the direction $w$ is 
\begin{equation*}
 -  D K_p(u) w - D^2 K_p (u)(v - u, w) + DK_p(u) w = \covat q {u-v}{w},
\end{equation*}
hence the direction of steepest increase is $w \propto (u - v)$. The partial derivative in $v$ in the direction $w$ is
\begin{equation*}
   DK_p(v) w - DK_p(u) w = \expectat r w - \expectat q w.
\end{equation*}
This quantity is strictly positive for $w = v -u  \ne 0$ because of the monotonicity of $K_p$.

The second partial derivative in $u$ in the direction $w_1,w_2$ is
\begin{equation*}
  D^3K_p(u)(u-v,w_1,w_2) + D^2K_p(u)(w_1,w_2),
\end{equation*}
which reduces on the diagonal $q=r$ to $D^2K_p(u)(w_1,w_2) = \covat q {w_1}{w_2}$. 

The second partial derivative in $v$ in the direction $w_1,w_2$ is
\begin{equation*}
D^2 K_p(v)(w_1,w_2) = \covat r {w_1}{w_2}
\end{equation*}
which reduces on the diagonal $q=r$ to $D^2K_p(u)(w_1,w_2) = \covat q {w_1}{w_2}$. Some approaches to Information Geometry are based on the Hessian on the diagonal of a divergence (yoke, potential) e.g., \cite{barndorff-nielsen|jupp:1997}, \cite{shima:2007}.
 
This is a case of of high regularity as we assume the densities $q$ and $r$ positive and connected by an open exponential arc. In our framework there is another option, namely to consider $(q,r) \mapsto D(q\Vert r)$ as a mapping defined on $\maxmix p \times \maxexp p$, see the Ex \ref{ex:pitagoreanth} below. Without any regularity assumption one can look for the joint semicontinuity as in \cite[Sec 9.4]{ambrosio|gigli|savare:2008}.
\end{example}

\begin{example}[Pitagorean theorem]\label{ex:pitagoreanth}

  Let $p \in \pdensities$ and $s_p : \maxexp p \to \Bspace p$ and $\eta_p : \maxmix p \to \mBspace p$ be charts respectively in the exponential and mixture manifold. We can exploit the duality between $\mBspace p$ and $Bspace p$ as follows. Let be given densities $q \in \maxexp p$, $u= s_p \left(q\right)$ and $r \in \maxmix p \cap \densities$.  We have 
\begin{equation*}
\expectat p { \eta_p \left(r\right) s_p \left(q\right) } = \expectat p { \left( \frac{r}{p} -1\right) u } = \expectat r u.
\end{equation*} 
As 
\begin{equation*}
  u = \log\left(\frac{q}{p}\right) - \expectat p { \log\left(\frac{q}{p}\right)} = 
  \log\left(\frac{q}{p}\right) + D \left( p \, \| \, q \right),
\end{equation*}
we have
\begin{equation*}
\expectat p { \eta_p \left(r\right) s_p \left(q\right) } = - D \left( r \, \| \, q \right) + D \left( r \, \| \, p \right)+D \left( p \, | \, q \right).
\end{equation*}
In particular, if the left side is zero, 
\begin{equation*}
D \left( r \, \| \, q \right) = D \left( r \, \| \, p \right)+ D \left( p \, \| \, q \right),
\end{equation*}
which is the Pitagorean relation in Information Geometry e.g., \cite{csiszar|matus:2003}.
\end{example}

\begin{example}[Stochastics]
On a Wiener space $(\Omega, \mathcal F, (\mathcal F_t)_{t \ge 0},\nu)$, the geometric Brownian motion is $Z_t = \expof{W_t - t/2}$ is strictly positive and $\int Z_t \ d\nu = 1$.  From Ito's formula, $\Phi=\cosh -1$, $\alpha > 0$,
\begin{equation*}
  \int \Phi(\alpha Z_t)\ d\nu = \frac{\alpha^2}2 \int_0^t dt\ \int \Phi(\alpha Z_s) \ d\nu.
\end{equation*}
It follows that $\int \Phi(\alpha Z_t)\ d\nu = \euler^{\alpha^2 t / 2} - 1$ if finite for all $\alpha$, hence $Z_t \in \maxexp 1$ and $s_1\left(Z_t\right) = W_t$. The statistical model $Z_t$, $t\ge 0$, is intrinsically non parametric because the vector space generated by the $W_t$, $t >0$, has $L^2$ closure equal to the full gaussian space $\setof{\int f \ dW}{\int_0^t f(s)^2 \ ds < \infty, t>0}$.

The exponential representation of $\pdensities$ is static, but a dynamic variant has been devised by Imparato \cite{imparato:thesis} and \cite{SST:2013}.
\end{example}

\subsection{Exponential families, parameterization}
\label{sec:exponential-families}
Both the mixture and the exponential manifold are intrinsic structures which are constructed with virtually no assumptions but the positivity of the densities. Specific applications will require special assumptions and special parameterization. We suggest to distinguish the manifold charts from other useful parameterization via  definitions of this type. Note that we have defined our statistical manifolds is such a way the coordinate space of each $p$-chart is identifiable with the tangent space at $p$. 
\begin{definition}
  Let $A$ be an open subset of the exponential manifold $\pdensities$ and let $\mathcal M$ be a manifold. An $k$-differentiable mapping $F \colon A \to \mathcal M$ is a proper parameterization if the tangent linear form $T_p F \colon T_p \pdensities \to T{F(p)}\mathcal M$ is surjective. 
\end{definition}
This approach is different from the widely used reverse approach, where a parameterization is a mapping from a parameter's manifold to set of densities. Given a proper parameterization, either the inverse tangent mapping is continuous, in which case the parameterization is actually a chart, or it is possible to pull back the image structure to $\pdensities$. It is what we have done to build the mixture manifold bases on the mean parameterization $q \mapsto q/p-1$.

A similar discussion applies when dealing with sub-manifolds. According to the general theory of Banach manifolds, a sub-manifold $\mathcal M$ is a subset of $\pdensities$ with a tangent space $T_p\mathcal M$ which splits in $T_p\pdensities = \Bspace p$, that is closed and has a complement space. Some closed subspaces of $\Bspace p$ split, e.g. finite dimensional subspaces.

In particular, we have the following definition of exponential family.
\begin{definition}
Let $V$ be a closed subspace of $\Bspace p$. The $V$-exponential family is the subset of the maximal exponential family $\maxexp p$ defined by
\begin{equation*}
  \mathcal E_V(p) = \setof{\euler^{u - K_p(u)}\cdot p}{ u \in \mathcal S_p \cap V}
\end{equation*}
\end{definition}
If a splitting of $V$ is known, then the exponential family is a sub-manifold of the exponential manifold. The proper definition of a sub-manifold in the framework of mixture/exponential manifold is an importan topic to be investigated beyond the partial results in the literature we hare summarizing here. 

\section{Differential equations}\label{sec:diff-eq}
The study of differential or evolution equations fits nicely in the theory of Banach manifolds \cite[Ch IV]{lang:1995}, but would require a generalizations to tackle the technical issue of the non reflexive duality of the couple $L^\Phi$, $L^{\Phi_*}$. We review here the language introduced in \cite{pistone:2009SL} and mention some examples.

 Let $p(t)$, $t \in I$ open, be a curve in a given maximal exponential family $\maxexp p$, with $u(t)$ the $p$-coordinate of $p(t)$, $p(t) = \expof{u(t) - K_p(u(t))}\cdot p$. If the curve $I \ni t \mapsto u(t) \in B_p$ is of class $C^1$ with derivative at $t$ denoted $\dot u(t)$, the mapping $p \colon I \ni t \mapsto p(t) \in \maxexp p$ is differentiable and its derivative is the element of $T_{p(t)} \pdensities$ whose coordinate at $p$ is $\dot u(t)$. Its velocity is defined to be the curve $I \to T\pdensities$ such that the $p$-coordinates of $(p(t), \delta p(t)) \in T_{p(t)} \pdensities$ are $(u(t), \dot u(t))$, that is $\dot u (t) = \delta p(t) - \expectat p {\delta p(t)}$ and $\delta p(t) = \dot u(t) - \expectat {p(t)} {\dot u(t)}$. We have
\begin{equation}
  \label{eq:diff-eq-1}
 \delta p(t) = \dot u(t) - \expectat {p(t)} {\dot u(t)} = \derivby t {(u(t) - K_p(u(t))} =
 \derivby t {\logof{\frac{p(t)}{p}}} = \frac {\dot p(t)}{p(t)},
\end{equation}
where the last equality is computed in $L^{\Phi_*}(p)$.

We can do a similar construction in the mixture manifold. Let $p(t)$, $t \in I$ open, be a curve in the mixture manifold $\sdensities$, $p(t) = (1 + v(t))p$. If the curve $I \ni t \mapsto v(t) \in \mBspace p$ is of class $C^1$, the velocity of the curve is $I \to (p,\delta p) \in T\sdensities$, with $p$-coordinates of $\delta p(t)$ equal to $\dot v(t)$, that is $\delta p(t) = \frac p{p(t)} \dot v(t)$. It follows again $\delta p(t) = {\dot p(t)}/{p(t)}$. The two representations of the velocity equal  in the \emph{moving frame}. Note that other representations would be possible and are actually used in the literature e.g.,
\begin{equation*}
  \frac{\dot p(t)}{p(t)} = \frac{\derivby t 2\sqrt{p(t)}}{\sqrt{p(t)}},
\end{equation*}
which is a representation to be discussed in Sec. \ref{sec:H-ESP} below, which is based on the embedding $\pdensities \ni p \mapsto \sqrt p \in L^2(\mu)$. 

A \emph{vector field} $F$ of of the exponential manifold $\pdensities$, is a section of the tangent bundle $T\pdensities$, $F(p) \in T_p\pdensities$, with domain $p \in \domainof F$. A curve $p(t)$, $t \in I$, is an \emph{integral curve} of $F$ if $p(t) \in \domainof F$ and $\delta p(t) = F\left(p(t)\right)$, $t \in I$. Same definition in the case of the mixture manifold. In the moving frame the differential equation can be written $\dot p(t) = F(p(t)) p(t)$. In the exponential chart at $p$ we write $\dot u(t) = \etransport {p(t)} p F(p(t))$ together with $p(t) = \euler^{u(t) - K_p(u)} \cdot p$. In the mixture chart at $p$ we write $\dot v(t) = \mtransport {p(t} p F(p(t))$ with $ p(t) = (1+v(t)) p$. 

We discuss briefly the existence of solutions. In the exponential chart we have the differential equation
\begin{align*}
\dot u(t) &= \etransport{p(t)} p F(p(t)) =  F(p(t)) - \expectat p { F(p(t))}, \\
               p(t) &= \euler^{u(t) - K_p(u(t))} \cdot p, \\
               u(0) &= 0.
\end{align*}
We can use the duality on $\Bspace p \times \mBspace p$ to check if the functional 
\begin{equation*}
  \bm F(u) = F\left(\euler^{u(t) - K_p(u(t))} \cdot p\right) - \expectat p {F\left(\euler^{u(t) - K_p(u(t))} \cdot p\right)}
\end{equation*}
satisfies a one-sided Lipschitz condition $\scalarof{\bm F(u) - \bm F(v)}{u - v} - \lambda \scalarof{u-v}{u-v} \le 0$, $\lambda > 0$. We have
\begin{equation*}
  \scalarof{\bm F(u) - \bm F(v)}{u - v} = \covat p {F(e_p(u)\cdot p)-F(e_p(v)\cdot p)}{u-v}.
\end{equation*}
The uniqueness of the solution follows easily from standard arguments for evolution equations. Proof of existence requires usually extra conditions in order to apply methods from functional analysis. We discuss a set of typical examples.

\begin{example}[One dimensional exponential and mixture families]
Let $f \in L^\Phi(p)$ and define the vector field $F$ whose value at each $q \in \maxexp p$ is represented in the frame at $q$ by $q \mapsto f - \expectat q f$. We can assume without restriction that $f \in \Bspace p$ in which case $f - \expectat q f = \etransport pq f$. The differential equation in the moving frame is $\dot p(t)/p(t) = f - \expectat {p(t)} f$, with solution $\logof{p(t)} = \logof{p(0)} + t\left(f - \expectat {p(t)} f\right)$. In the fixed frame at the initial condition $p(0) = p$ the equation is $\dot u(t) = f - \expectat p f$, with solution $u(t) = t(f-\expectat p f)$, hence $p(t) = \expof{t(f-\expectat p f) - K_p\left(t(f-\expectat p f)\right)}\cdot p$. The equation in the mixture manifold, $f \in \mBspace p$ is $\dot p(t)/p(t) = \mtransport p{p(t)} f$, with solution $p(t) = p(1+tf)$. We have constructed here the geodesics of the two manifolds.
\end{example}
\begin{example}[Optimization]
Stochastic relaxation of optimization problems using tools from Information Geometry has been studied in \cite{malago|matteucci|dalseno:2008}, \cite{malago|matteucci|pistone:2011a}, \cite{malago|matteucci|pistone:2011b}, \cite{arnoldetal:2011arXiv}, \cite{malago:2012thesis},\cite{malago|matteucci|pistone:2013CEC}.  The expectation of a real function $F \in L^\Phi(p)$ is an affine function in the mixture chart, $\expectat q F = \expectat p {F \left(\frac{q}{p}-1\right)} + \expectat p F$, while in the exponential chart is a function of $u = s_p(q)$, $\widetilde F(u) = \expectat q F$. The equation for the derivative of the cumulant function $K_p$ gives 
\begin{equation*}
  \widetilde F(u) = \dirderat {K_p} u {(F - \expectat p F)} + \expectat p F,
\end{equation*}
and the derivative of this function in the direction $v$ is the Hessian of $K_p$ applied to $(F - \expectat p F)\otimes v$:
\begin{equation*}
  \dirderat \Phi u v = D^2 K_p(u) (F - \expectat p F)\otimes v = \covat q v F.
\end{equation*}
The direction of steepest ascent of the expectation is $F - \expectat q F$, hence the equation of the flow is $\delta p(t) = \etransport p{p(t)} F$, whose solution i the exponential family with canonical statistics $F$ . In practice, the flow is restricted to an exponential model $\mathcal E_V(p)$, $V \subset \Bspace p$ and the direction of steepest ascent is a projection of $F$ onto the $V$, if it exists.
\end{example}

\begin{example}[Heat equation]
The heat equation $\frac{\partial}{\partial t} p(t,x) - \frac{\partial^2}{\partial x^2} p(t,x) = 0$, $x \in \reals$ for simplicity, is an example of evolution equation in $T\pdensities$ with vector field
\begin{equation*}
  F(p)(x) = \frac{\frac{\partial^2}{\partial x^2} p(x)}{p(x)}.
\end{equation*}
A proper discussion would require an extension of our construction to Sobolev-Orlicz spaces \cite[Ch 8]{adams|fournier:2003} and  the solution would be based on the variational form of the heat equation.  For each $v$ with the proper domain of regularity $D$
\begin{equation*}
  \expectat p {F(p) v} = \int p''(x) v(x) \ dx = - \int p'(x) v'(x) \ dx = - \expectat p {\frac {p'}{p} v'}
\end{equation*}
from which the weak form of the evolution equation follows,
\begin{equation*}
  \expectat {p(t)} {\delta p(t) v} + \expectat {p(t)} {F_0(p(t)) v} = 0, \quad v \in D,
\end{equation*}
where $F_0(p) = \nabla p/p$ is the vector field associated to the translation model $p_\theta (x) = p(x - \theta)$, see e.g, \cite{otto:2001}.
\end{example}
 
\begin{example}[Decision theory]
A further interesting example of evolution equation arises in decision theory \cite{parry|dawid|lauritzen:2012}. For simplicity the sample space is $\reals$, $q \in \maxexp p$, $q = \euler^{u-K_p(u)}\cdot p$ and $\log q - \log p = u - K_p(u)$. Assume $u$ belongs to the Sobolev-Orlicz space
\begin{equation*}
  W_0^{\Phi,1} = \setof{u \in L_0^{\Phi}(p)}{\nabla u \in L^\Phi(p)},
\end{equation*}
where $\nabla$ denotes the spatial derivative. The following expression is a statistical divergence
   \begin{align*}
     d(p,q) &= \frac14\expectat p {\absoluteval{\nabla \log q - \nabla \log p}^2}\\
 &= \frac14\expectat p {\absoluteval{\nabla u}^2}.
   \end{align*}
For $u,v_0 \in W^{\Phi,1}$ we have a bilinear form
\begin{align*}
  (u,v) &\mapsto \expectat p {\nabla u \nabla v} = \int u_x(x) v_x(x) p(x) dx, \\ &= - \int \nabla (u_x(x) p(x)) v(x) dx, \\  &= - \int (\Delta u(x) p(x) + \nabla u(x) \nabla p(x)) v(x) dx, \\ &= \expectat p {(-\Delta u - \nabla \log p \nabla u) v},
\end{align*}
where $\Delta$ is the second derivative in space. We have
\begin{equation*}
  \expectat p {\nabla u \nabla v} = \expectat p {F_p (u) v}, \quad F_p (u) = - \Delta u - \nabla \log p \nabla u,
\end{equation*}
with $F_p(u)$ in a proper Sobolev-Orlicz space $\prescript{*}{} W_0^{\Phi,1}$. This provides a classical setting for a weak form of evolution equation. The mapping $q \mapsto \Delta \log q / \log q$ is represented for $q=\euler^{u - K_p(u)}\cdot p$ by
    \begin{equation*}
      u \mapsto \frac{\Delta (u - K_p(u)) + \Delta \log p}{(u - K_p(u)) + \log p}.
    \end{equation*}
The mapping $q \mapsto \absoluteval{\nabla \log q}^2$ is represented by
\begin{equation*}
u \mapsto \absoluteval{\nabla u}^2.
\end{equation*}
\end{example}

\begin{example}[Boltzmann equation] Orlicz spaces as a setting for Boltzmann equation has been recently suggested by \cite{majewski|labuschagne:2013arXiv:1302.3460}. We consider the space-homogeneous Boltzmann equation see e.g., \cite{villani:2002review}. On the sample space $(\reals^3,d \bm v)$ let $f_0$ be the standard normal density. For each $f \in\maxexp{f_0}$ we define the Boltzmann operator to be
  \begin{multline*}
Q(f)(\bm v) = \\ \int_{\reals^3}\int_{S^2} (f(\bm v - \bm x \bm {x'}(\bm v - \bm w))f(\bm w + \bm x \bm {x'}(\bm v - \bm w))-f(\bm v)f(\bm w))\absoluteval{\bm {x'}(\bm v - \bm w)}\ d\bm x\ d\bm w,    
  \end{multline*}
where $'$ denotes the transposed vector, $S^2$ is the unit sphere $\setof{\bm x \in \reals^3}{\bm x'\bm x = 1}$, $d\bm x$ is the surface measure on $S^2$. The $\reals^{(3+3)\times(3+3)}$ matrix
\begin{equation*}
  A \colon \left\{\begin{aligned}
    \bm v_* &= \bm v - \bm x \bm {x'}(\bm v - \bm w) = (I - \bm x \bm {x'})\bm v + \bm x \bm {x'} \bm w , \\
    \bm w_* &= \bm w + \bm x \bm {x'}(\bm v - \bm w) = \bm x \bm {x'} \bm v + (I - \bm x \bm {x'})\bm w
  \end{aligned}\right.
\end{equation*}
is such that $AA$ is the identity on $\reals^6$, in particular $\det A = \pm 1$, and $\bm{x'}(\bm v - \bm w) = - \bm{x'}(\bm v_* - \bm w_*)$. Hence the measure $\absoluteval{\bm {x'}(\bm v - \bm w)}\ d\bm v\ d\bm w$ is invariant under $A$. The integral of the Boltzmann operator is zero:
  \begin{multline*}
\int_{\reals^3} Q(f)(\bm v)\ d\bm v= \\ \int_{S^2} \int_{\reals^3} \int_{\reals^3}(f(\bm v_*)f(\bm w_*)-f(\bm v)f(\bm w))\absoluteval{\bm {x'}(\bm v - \bm w)}\ d\bm w\ d\bm v\ d\bm x = \\
\int_{S^2} \int_{\reals^3} \int_{\reals^3}f(\bm v_*)f(\bm w_*)\absoluteval{\bm {x'}(\bm v_* - \bm w_*)}\ d\bm w_*\ d\bm v_*\ d\bm x - \\ \int_{S^2} \int_{\reals^3} \int_{\reals^3}f(\bm v)f(\bm w))\absoluteval{\bm {x'}(\bm v - \bm w)}\ d\bm w\ d\bm v\ d\bm x = 0.   
  \end{multline*}
Note that $\bm{v'}\bm v + \bm{w'}\bm w = \bm{v_*'}\bm v_* + \bm{w_*'}\bm w_*$, hence 
\begin{equation*}
f_0(\bm v)f_0(\bm w) = (2\pi)^3 \euler^{-(1/2)(\bm{v'}\bm{v}+\bm{w'}\bm{w})} = f_0(\bm v_*)f_0(\bm w_*).
\end{equation*}
If we write $f(\bm v) /f_0(\bm v) = g(\bm v)$, the Boltzmann operator takes the form
  \begin{multline*}
Q(f)(\bm v) = \\ f_0(\bm v) \int_{\reals^3}\int_{S^2} (g(\bm v_*)g(\bm w_*)-g(\bm v)g(\bm w))f_0(\bm w)\absoluteval{\bm {x'}(\bm v - \bm w)}\ d\bm x\ d\bm w = \\ F_0(f)(\bm w) f_0(\bm v) ,    
  \end{multline*}
and $\expectat{f_0}{F(f)} = 0$ i.e. both  $\etransport {f_0}f F_0(f)$ and $\mtransport {f_0}f F_0(f)$ are candidate for a vector field in the exponential manifold.
\end{example}

\section{The Hilbert bundle}
\label{sec:hilbertbundle}
To each positive density $p \in \pdensities$ we attach the Hilbert space of centered square-integrable random variables $H_p = L^2_0(p)$ in order to define the a vector bundle $H \pdensities$ on the set $\setof{(p,u)}{p \in \pdensities, u \in H_p}$. If the densities $p$ and $q$ both belong to the same maximal exponential family $\mathcal E$, then according to Prop. \ref{th:22} we know the Banach spaces $L^\Phi(p)$ and $L^\Phi(q)$, $\Phi = \cosh -1$, to be equal as sets and have equivalent norms. The subspaces $\Bspace p$ and $\Bspace q$, are continously embedded, respectively, into the Hilbert spaces $L^2_0(p)$ and $L^2_0(q)$. Moreover, $\etransport p q \colon B_p \ni u \mapsto u - \expectat q u \in B_q$ is an isomorphism. Under the same condition $p,q \in \mathcal E$, $L^2_0(p)$ and $L^2_0(q)$ are continuously embedded, respectively, into the Banach spaces $\mBspace p = L^{\Phi_*}_0(p)$ and $\mBspace q = L^{\Phi_*}_0(q)$, which admit the isomorphism $\mtransport p q \colon \mBspace p \ni u \mapsto \frac pq u \in \mBspace q$. All spaces are embedded subspaces of the space of measurable random variables $L^0(\mu)$, see the diagram \eqref{eq:CD0}. The isomorphism $\transport p q \colon H_p \to H_q$ is to be defined in the next sections.
\begin{equation}\label{eq:CD0}
\xymatrix{%
L^\Phi(p)\ar@{=}[d]&\Bspace p \ar[l]\ar[d]_{\etransport p q}\ar[r]&H_p\ar[d]_{\transport pq}\ar[r]&\mBspace p \ar[d]_{\mtransport p q}\ar[r]& L^0(\mu) \\
L^\Phi(q)&\Bspace q \ar[l]\ar[r]&H_p\ar[r]&\mBspace q\ar[ur]} 
\end{equation}

\begin{example}\label{ex:notequalHs}
An example shows that $p,q \in \mathcal E$ does not imply $L^2(p)=L^2(q)$: in the exponential model $p_\theta =\theta \euler^{-\theta x}$ on $(\reals_>,dx)$, $\theta > 0$, the random variable $v(x) = \euler^{x/4}$ belongs to $L^2(p_\theta)$ if, and only if, $\theta > 1/2$. The equatity is not generally true even locally, unless we restrict ourself to cases were the steepness condition holds. If $v$ belongs to both $L^2(p)$ and $L^2(q)$, it belong to $L^2(r)$ with $r$ in the closed exponential arc between $p$ and $q$, bat the convex function $B_p \ni u \mapsto \int v^2 \euler^u p \ d\mu$ is finite at zero, but could take a $+\infty$ value on any neigborhood of 0. To construct an example, consider the nonsteep distribution already used in Ex. \ref{ex:nonsteep}. For the reference measure $\mu(dx) = (1+x)^{-\frac32}\ dx$, we rewrite Eq. \eqref{eq:upgamma-int} as
\begin{equation*}
  \int_0^\infty \euler^{-\theta x}\ \mu(dx) =
 \begin{cases}
   \sqrt\theta \euler^{\theta} \upgamma{-\frac12}{\theta} & \text{if $\theta > 0$,} \\
2 & \text{if $\theta = 0$}, \\ +\infty & \text{if $\theta < 0$.} 
 \end{cases}
\end{equation*}
The exponential family $p_\theta \propto \euler^{- \theta x}$ is defined for $\theta > 0$. The random variable $u(x) = \euler^{x/2}$ has second moment
\begin{equation*}
  \int_0^\infty (u(x))^2 p_\theta(x)\ \mu(dx) \propto \int_0^\infty \euler^{-(\theta-1)x}\ \mu(dx),   
\end{equation*}
which is finite for $\theta \ge 1$ and infinite for $0 < \theta < 1$.
\end{example}

We are going to show that the $H_p$'s are actually isomorphic as Hilbert spaces and that our Hilbert bundle can be viewed as a push-back of the tangent bundle of the unit sphere of $L^2(\mu)$. In turn, this construction provides a derivation of the metric connection, see \cite[VIII \S 4]{lang:1995}. Connections on statistical manifolds are a key ingredient of Amari's theory \cite{amari|nagaoka:2000}, while the non parametric version has been done in \cite{gibilisco|pistone:98} and \cite{gibilisco|isola:1999} in the $L^p$ case, commutative and non commutative, respectively. Cfr. also the critical discussion in \cite{grasselli:2010AISM}.The construction here is different. In order to have a clear cut distinction between the geometric Hilbert i.e. $L^2(\mu)$ case and its application to statistical manifolds, we use a bold face notation for points and vectors in the former case.
\subsection{The sphere of $L^2(\mu)$}
\label{sec:sphere}
The unit sphere $S = \setof{\bm x \in L^2(\mu)}{\int \bm x^2 \ d\mu}$ is a Riemannian manifold with tangent bundle $T S = \setof{(\bm x, \bm u)}{\bm x \in S, \bm u \in \set{\bm x}^\perp}$ and metric $g_{\bm x}(\bm u, \bm v) = \int \bm u \bm v \ d\mu = \scalarof{\bm u}{\bm v}$. We will use the projection charts $s_{\bm x} (\bm y) = \Pi_{\bm x} \bm y = \bm y - \scalarof {\bm x}{\bm y} \bm x$ with domain $\setof{\bm y \in S}{\scalarof {\bm x}{\bm y} > 0}$ and codomain $\setof{\bm u \in T_{\bm x} S}{\scalarof{\bm u}{\bm u} < 1}$. The patch is $s_{\bm x}^{-1}(\bm u) = \bm u + \sqrt{1 - \scalarof{\bm u}{\bm u}^2} \bm x$.

\begin{proposition}\label{prop:Hisometry}
For $\bm x, \bm y \in S$ and $\bm u \in T_{\bm x} S$, define
 \begin{equation}\label{eq:S-transport}
    \transport{\bm x}{\bm y} \bm u = \bm u  -  \left(1 +  \scalarof{\bm x}{\bm y}\right)^{-1} \scalarof{\bm u}{\bm y}(\bm x + \bm y).
  \end{equation} 
\begin{enumerate}
\item $\transport{\bm x}{\bm y} \bm u \in T_{\bm y} S$ and $\transport{\bm y}{\bm x} \circ \transport{\bm x}{\bm y} \bm u = \bm u$.
\item For $\bm u, \bm v \in T_{\bm x} S$ the isometric property $\scalarof{\transport{\bm x}{\bm y} \bm u}{\transport{\bm x}{\bm y} \bm v} = \scalarof{\bm u}{\bm v}$ holds, hence
\begin{equation}
  \label{eq:Uvsg}
  g_{\bm y}(\transport{\bm x}{\bm y} \bm u,\transport{\bm x}{\bm y} \bm v) = g_{\bm x}(\bm u,\bm v).
\end{equation}
\end{enumerate}
\end{proposition}
\begin{proof}
The formula for $\transport {\bm x}{\bm y}$ is obtained by splitting $\bm u$ into a component orthogonal to both $\bm x$ and $\bm y$, which is left invariant, and rotating the other component in the plane generated by $\bm x$ and $\bm y$. Note that $\transport{\bm z}{\bm y} \circ \transport{\bm y}{\bm x} \ne \transport{\bm z}{\bm x}$ unless $\bm z$ belong to the plane generated by $\bm x$ and $\bm y$. In a full definition, the trasport should be associated to a specific path, but we do not discuss here this point. \qed
\end{proof}
\begin{example}
Let $\mu$ the the standard normal distribution and let $H_n$, $n = 0, 1, \dots$ the Hermite polinomials: $H_0(x) = 1$, $H_1(x) = x$, $H_2(x) = x^2 -1$, \dots, see \cite[V.1]{malliavin:1995}. The Hermite polynomials form an orthogonal basis of $L^2(\mu)$, hence $(H_n)_{n \ge 1}$ is an orthogonal basis of $T_1 S = L^2_0(\mu)$. If $\expectof{Y^2} = 1$, the sequnce
\begin{equation*}
  \transport1Y H_n = H_n - (1+ \expectof Y)^{-1} \expectof{YH_n} (1 + Y), \quad n = 1, 2, \dots
\end{equation*}
is an orthogonal basis of $T_Y S$.
\end{example}
The isometric affine transport in \eqref{eq:S-transport} provides charts for the tangent bundle $TS$: given $\bm x \in S$, for each $y \in S$, $\scalarof{\bm x}{\bm y} > 0$, and $\bm v \in T_{\bm y}$, then the coordinates of $(\bm y, \bm v) \in T_{\bm y} S$ are 
\begin{equation}\label{eq:TSchart}
  s_{\bm x}(\bm y,\bm v) = \left(\Pi_{\bm x}\bm y, \transport{\bm y}{\bm x} \bm v\right) \in T_{\bm x} S \times T_{\bm x} S,
\end{equation}
where $\Pi_{\bm x}\bm y = \bm y - \scalarof {\bm x}{\bm y} \bm x$ is the orthogonal projection on $T_{\bm x} S$. The transition map from $\bm x_1$ to $\bm x_2$ is
\begin{multline*}
  T_{\bm x_1} S \times T_{\bm x_1} S \ni (\bm u, \bm v) \mapsto \\ \left(\Pi_{\bm x_2}\left(\bm u + \sqrt{1 - \scalarof{\bm u}{\bm u}^2}  \bm x_1 \right),\transport{\bm x_1}{\bm x_2} \bm v\right) \in T_{\bm x_2} S \times T_{\bm x_2} S.
\end{multline*}

\subsection{Covariant derivative on S}
\label{sec:covariant-derivative}
Let $F$ be a vector field on the sphere $S$ and let $\bm x(t)$, $t \in I$ be a curve on $S$, $\bm x(0) = \bm x$. As $(\bm x(t),F(\bm x(t)) \in T_{\bm x(t)}$, in the chart at $\bm x$ we have
\begin{equation*}
s_{\bm x}(\bm x(t),F(\bm x(t)) = \left(\Pi_{\bm x} \bm x(t), \transport{\bm x(t)}{\bm x} F(\bm x(t)) \right). 
\end{equation*}
We assume $t \mapsto \bm x(t)$ is differentiable in $L^2(\mu)$, so that $\left. \derivby t \Pi_{\bm x} \bm x(t) \right|_{t=0} =  \Pi_{\bm x} \bm{\dot x}(0) = \bm{\dot x}(0)$. The derivative with respect to $\bm x$ of $\Pi_{\bm x} \bm y$ in direction $\bm w$ is
\begin{equation*}
d_{\bm w}\left(\bm x \mapsto \Pi_{\bm x} \bm y\right) = - \scalarof{\bm y}{\bm w} \bm x - \scalarof{\bm y}{\bm x} \bm w.
\end{equation*}
The derivative with respect to $\bm x$ of $\transport{\bm x}{\bm y} \bm u$ in direction $\bm w$ is
\begin{equation*}
  d_{\bm w}\left(\bm x \mapsto \transport{\bm x}{\bm y} \bm u\right) = \left(1 +  \scalarof{\bm x}{\bm y}\right)^{-2} \scalarof{\bm w}{\bm y}\scalarof{\bm u}{\bm y}(\bm x + \bm y) - \left(1 +  \scalarof{\bm x}{\bm y}\right)^{-1} \scalarof{\bm u}{\bm y}\bm w,
\end{equation*}
so that $\left. d_{\bm w}\left(\bm x \mapsto \transport{\bm x}{\bm y}u\right) \right|_{\bm y = \bm x} = 0$ because $\scalarof{\bm u}{\bm x} = 0$.

Let $F$ be a vector field on the sphere $S$, and assume $F$ is the restriction of a smooth $L^2(\mu)$-valued function, defined of a neighborhood of $S$, with directional derivative denoted $d_{\bm w} F(\bm x)$. Let $\bm x(t)$, $t \in I$ be an  $L^2(\mu)$-smooth curve on $S$, $\bm x(0) = \bm x$, $\bm{\dot x}(0) = \bm w \in T_{\bm x} S$. As we want to compute $\derivby t \transport{\bm x(t)}{\bm x} F(\bm x(t))$, we write $\transport{\bm x(t)}{\bm x} F(\bm x(t)) = \transport{\bm x(t)}{\bm x} \Pi_{\bm x(t)}F(\bm x(t))$, with $\Pi_{\bm z} \bm f = \bm f - \scalarof{\bm f}{\bm x(t)} \bm x(t)$ and $d_{\bm w}(\bm z \mapsto  \Pi_{\bm z} \bm f =  - \scalarof{\bm f}{\bm w} \bm x - \scalarof{\bm f}{\bm x} \bm w = - \scalarof{\bm f}{\bm w} \bm x$ if $\bm f \in T_{\bm x} S$. From the previous computations,
\begin{align}
  \left. \derivby t \transport{\bm x(t)}{\bm x} F(\bm x(t)) \right|_{t=0} &= \left. \derivby t \transport{\bm x(t)}{\bm x} \Pi_{\bm x(t)} F(\bm x(t)) \right|_{t=0} \notag \\ &= -\scalarof{F(\bm x)}{\bm w} \bm x + d_{\bm w} F(\bm x) - \scalarof{d_{\bm w} F(x)}{\bm x} \bm x \notag \\
&= -\scalarof{F(\bm x)}{\bm w} \bm x + \Pi_{\bm x} d_{\bm w} F(\bm x). \label{eq:DS}
\end{align}

Let $F,G,W$ be smooth vector fields on the sphere $S$. From \eqref{eq:DS} we can compute the metric derivative $\nabla_W F$, i.e. the unique covariant derivative such that
\begin{equation*}
  D_W g(F,G) = g(D_W,G) + g(F,D_W G),
\end{equation*}
see \cite[VIII \S 4]{lang:1995}.
\begin{proposition}
The value of the metric derivative $D_W F$ at $\bm x \in S$ is
\begin{equation*}
  D_{W(\bm x)} F(\bm x) = d_{W(\bm x)} F(\bm x) - \scalarof{d_{W(\bm x)}F(\bm x)}{\bm x} \bm x = \Pi_{\bm x} d_{W(\bm x)} F(\bm x).
\end{equation*}
\end{proposition}
\begin{proof}
Let $\bm x(t)$, $t \in I$ be a smooth curve on $S$, such that $\bm{\dot x}(t) = W(\bm x(t))$, $\bm x(0) = \bm x$, $\bm {\dot x}(0) = W(\bm x) = \bm w$. Note that the first term in \eqref{eq:DS} is orthogonal to $T_{\bm x} S$.
\begin{multline*}
\left. \derivby t g_{\bm x(t)}(F(\bm x(t)),G(\bm x(t))) \right|_{t=0} = \left. \derivby t \scalarof{\transport{\bm x(t)}{\bm x} F(\bm x(t))}{\transport{\bm x(t)}{\bm x} G(\bm x(t))} \right|_{t=0} \\
= \scalarof{\left. \derivby t  \transport{\bm x(t)}{\bm x} F(\bm x(t)) \right|_{t=0}}{G(\bm x)} + \scalarof{ F(\bm x(t))}{\left. \derivby t \transport{\bm x(t)}{\bm x} G(\bm x(t)) \right|_{t=0}} \\
= g_{\bm x}(\Pi_{\bm x} d_{W(\bm x)} F(\bm x),G(\bm x)) + g_{\bm x}(F(\bm x),\Pi_{\bm x} d_{W(\bm x)} G(\bm x)).
\end{multline*}
\qed
\end{proof}

\subsection{The Hilbert bundle of the exponential manifold}
\label{sec:H-ESP}
For each density $p \in \pdensities$ the linear mapping $H_p \ni w \mapsto \bm{w\sqrt p}$ is an isometry onto $L^2(\mu)$ that maps $H_p$ onto $T_{\bm{\sqrt p}} S$. In fact $\int (\bm{w\sqrt p})^2\ d\mu = \expectat p {w^2}$ and $\scalarof{\bm{w\sqrt p}}{\bm{\sqrt p}} = \expectat p w = 0$. Viceversa, if $\bm y \in T_{\bm x} S$, then $\expectat p {(\bm x / \sqrt p)^2} = \scalarof{\bm x}{\bm x}$. In this case the embedding $\pdensities \ni p \mapsto \bm{\sqrt p}$ is an injection into the sphere $S$ of $L^2(\mu)$ and the sphere is smooth. It is the embedding used in \cite{amari|nagaoka:2000} that we discuss here in the framework of Banach manifolds, see the diagram \eqref{eq:CD1}. Applications of the non parametric setting are in e.g. \cite{brigo|legland|hanzon:99}.
\begin{equation}\label{eq:CD1}
\xymatrix{T\pdensities\ar[d]_{\pi}\ar[r]&H\pdensities\ar[d]_{\pi}\ar[r]^{w \mapsto \bm{w\sqrt p}}&TS\ar[d]^{\pi} \\
\pdensities\ar@{=}[r]&\pdensities\ar[r]^{p\mapsto\bm{\sqrt p}}&S\ar[r]&L^2(\mu)}  
\end{equation}
\begin{proposition}
  The mapping $\pdensities \ni p \mapsto \bm{\sqrt p} \in S$ is $C^\infty$ with derivative at $p$ in the direction $w \in T_p \pdensities$ equal to $\frac12 \bm{w \sqrt p} \in T_{\bm{\sqrt p}} S$.
\end{proposition}
\begin{proof}
Consider the mapping $\pdensities \ni p \mapsto \bm{\sqrt p} \in S$ in the charts at $p$ and $\sqrt p$, respectively. We go from $u \in \mathcal S_p$ to $S$ with
\begin{equation*}
u \mapsto q = \expof{u - K_p(u)} \cdot p \mapsto \bm{\sqrt q} = \expof{\frac12 u - \frac12 K_p(u)} \sqrt p
\end{equation*}
and to $T_{\bm{\sqrt p}} S$ with
\begin{multline*}
u \mapsto \sqrt q - \int \sqrt{pq}\ d\mu \sqrt p = \\ \left(\expof{\frac12 u - \frac12 K_p(u)} - \expectat p {\expof{\frac12 u - \frac12 K_p(u)}}\right) \sqrt p.
\end{multline*}
The mapping $u \mapsto \euler^{u/2}$ is analytic from the open unit ball of $\Bspace p$ to $H_p=L^2_0(p)$ according to prop. \ref{prop:expisanalytic}; multiplication by $\sqrt p$ is an isometry of Hilbert spaces. The real function $u \mapsto K_p(u)$ is infinitely Fr\'echet differentiable according to Prop. \ref{cor:cumulant}. The derivative is computed as 
\begin{multline*}
  \left.d_w(u \mapsto \expof{\frac12 u - \frac12 K_p(u)})\right|_{u=0} =  \\ \left.\expof{\frac12 u - \frac12 K_p(u)}\left(\frac12 w - d_w K_p(u)\right)\right|_{u=0} = \frac12 w,
\end{multline*}
and finally applying the isometry $\frac12 w \mapsto \frac12 w \sqrt p$.
\qed
\end{proof}

For each $p \in \pdensities$ define $I_p \colon H_p \pdensities \ni u \mapsto \sqrt p u \in T_{\bm{\sqrt p}} S$. We can use the isometry $I_p$ and the isometry $\transport{\bm x}{\bm y}$ of Prop. \ref{prop:Hisometry} to build an isometry
\begin{equation*}
  \transport pq = I_q^{-1} \circ \transport{\bm{\sqrt p}}{\bm{\sqrt q}} \circ I_p \colon H_p\pdensities \to H_q \pdensities.
\end{equation*}
as in the diagram \eqref{eq:CD2}.
\begin{equation}\label{eq:CD2}
\xymatrix{T_{\sqrt p} S\ar[r]^{\transport {\bm{\sqrt p}} {\bm{\sqrt q}}}&T_{\sqrt q}S\ar[d]^{\bm v \mapsto q^{-1/2}v}\\
                  H_p\pdensities\ar[u]^{u \mapsto \bm{p^{1/2}u}}\ar[r]_{\transport pq}&H_q\pdensities} \qquad \transport pq u = q^{-1/2} \transport{\bm{\sqrt p}}{\bm{\sqrt q}} (\bm{p^{1/2} u})
            \end{equation}

Substituting  $\bm u = \sqrt p u$, $\bm x = \sqrt p$, $\bm y = \sqrt
q$ in \eqref{eq:S-transport},
\begin{multline*}
  \bm u  -  \left(1 +  (\bm x \cdot \bm y)\right)^{-1} (\bm x + \bm y)
  (\bm u \cdot  \bm  y) = \\ \sqrt p u  -  \left(1 +  \int \sqrt{pq}\
    d\mu\right)^{-1} (\sqrt p + \sqrt q) (\int \sqrt{pq} u \ d\mu) = \\ \sqrt p u  -  \left(1 +  \expectat q {\sqrt{\frac pq}}\right)^{-1} (\sqrt p + \sqrt q) \expectat q {\sqrt{\frac pq}u}
\end{multline*}
so that
\begin{equation}\label{eq:myUonT}
 \transport pq u = \sqrt{\frac pq} u  -  \left(1 +  \expectat q {\sqrt{\frac pq}}\right)^{-1} \left(1 + \sqrt{\frac pq}\right) \expectat q {\sqrt{\frac pq}u}
\end{equation}
\begin{proposition}\label{prop:Tisometry}\ 
  \begin{enumerate}
  \item The mapping $\transport pq$ of Eq. \eqref{eq:myUonT} is an isometry of $H_p\pdensities$ onto $H_q\pdensities$. 
  \item $\transport qp \circ \transport pq u = u$, $u \in H_p\pdensities$ and $(\transport pq)^t = \transport qp$.
  \end{enumerate}
\end{proposition}
\begin{proof}\ 
We double-check the image:
    \begin{equation*}
      \expectat q {\transport pq u} = \expectat q {\sqrt{\frac pq} u}  -  \left(1 +  \expectat q {\sqrt{\frac pq}}\right)^{-1} \expectat q {1 + \sqrt{\frac pq}} \expectat q {\sqrt{\frac pq}u} = 0.
    \end{equation*}
We double-check the isometry:
\begin{multline*}
  \expectat q {(\transport pq u)^2} = \\ \expectat q {\frac pq u^2} - 2 \left(1 +  \expectat q {\sqrt{\frac pq}}\right)^{-1} \expectat q {\sqrt{\frac pq}u} \expectat q {\sqrt{\frac pq} u\left(1 + \sqrt{\frac pq}\right)}\\ + \left(\left(1 +  \expectat q {\sqrt{\frac pq}}\right)^{-1} \expectat q {\sqrt{\frac pq}u} \right)^2 \expectat q {\left(1 + \sqrt{\frac pq}\right)^2} \\ = \expectat p {u^2} - 2 \left(1 +  \expectat q {\sqrt{\frac pq}}\right)^{-1} \expectat q {\sqrt{\frac pq}u} \expectat q {\sqrt{\frac pq} u} \\ + \left(\left(1 +  \expectat q {\sqrt{\frac pq}}\right)^{-1} \expectat q {\sqrt{\frac pq}u} \right)^2 \left(2 + 2 \expectat q {\sqrt{\frac pq}}\right) \\
= \expectat p {u^2}.
\end{multline*}
\end{proof}

We can now define an atlas on the vector bundle $H\pdensities$ where the coordinates of $(q,v)$ are defined for $q \in \mathcal E p$ and $v \in H_q$ as
\begin{equation*}
s_p(q,v)  = (s_p(q),\transport qp v) \in \mathcal S_p \times H_p.
\end{equation*}

\subsection{Metric derivative in the Hilbert bundle}
\label{sec:metric-derivative}
Let $(p(t),F(t))$, $t \in I$, be a curve in $T\pdensities$, i.e. $p(t) \in \pdensities$ and $F(t) \in T_{p(t)} \pdensities$. Note that $\transport{p(t)}p F(t) \in T_p \pdensities$. We write $p(0)=p$, $\delta p(t) = \derivby t {\logof{p(t)}}$, $\delta p(0) = w$ and we compute $\left.\derivby t \transport{p(t)}p F(t)\right|_{t=0}$ from
\begin{multline}
  \transport{p(t)}p F(t) = \\ \sqrt{\frac{p(t)}{p}} F(t)  -  \left(1 +  \expectat p {\sqrt{\frac{p(t)}{p}}}\right)^{-1} \left(1 + \sqrt{\frac{p(t)}{p}}\right) \expectat p {\sqrt{\frac{p(t)}{p}} F(t)}. \label{eq:UF}
\end{multline}

The derivative of the first term in \eqref{eq:UF} is
\begin{align*}
\derivby t \left({\sqrt{\frac{p(t)}{p}}F(t)}\right) &= p^{-1/2} \derivby t \left({{p(t)}^{1/2}F(t)}\right) \\ &= p^{-1/2} \left(\frac12 p(t)^{-1/2} \dot p(t) F(t) + p(t)^{1/2} \dot F(t)\right) \\ &= \sqrt{\frac{p(t)}{p}} \left(\dot F(t) + \frac12 F(t) \delta p(t)\right),
\end{align*}
so that the derivative of the last factor is
\begin{equation*}
\derivby t \expectat p {\sqrt{\frac{p(t)}{p}}F(t)} = \expectat p {\sqrt{\frac{p(t)}{p}} \left(\dot F(t) + \frac12 F(t) \delta p(t)\right)}.   
\end{equation*}
Note that $\expectat p {\sqrt{\frac{p(0)}{p}}F(0)} = \expectat p {F(0)} = 0$, while
\begin{equation*}
\sqrt{\frac{p(0)}{p}} \left(\dot F(0) + \frac12 F(0) \delta p(0)\right) = \dot F(0) + \frac12 F(0) w.
\end{equation*}
In conclusion,
\begin{equation} \label{eq:metricderiv}
  \left.\derivby t \transport{p(t)}p F(t)\right|_{t=0} = \dot F(0) + \frac12 F(0) w - \expectat p {\dot F(0) + \frac12 F(0) w}.
\end{equation}

Note that in \eqref{eq:metricderiv} the term $F(0) w$ is the ordinary product of a random variable $F(0) \in H_p = L^2_0(p)$ and a random variable $w \in B_p = L^\Phi(p)$. In order to define a covariant derivative of the Hilbert bundle $H\pdensities$ we want $F(0) w \in L^2(p)$. For example, his would be true if $F$ were a vector field of the tangent space $T\pdensities$. 

\begin{definition}\label{def:metricderiv}
Let $G$, $F$ be vector fields in $H \pdensities$, i.e. $F(p), G(p) \in H_p \pdensities$. We define $D_G F$ to be the vector field defined by $D_G F (p) = \left.\derivby t \transport{p(t)}p F(t)\right|_{t=0}$, where $p(t)$ is a curve such that $p(0) = p$ and $\delta p(0) = G(p)$.  
\end{definition}

We conclude this section by summarizing the previous discussion in a statement.
\begin{proposition}
\begin{enumerate}
\item $D_G F$ in Def. \ref{def:metricderiv} is a covariant derivative.
\item 
Let $F_1$, $F_2$, $G$ be vector fields in $H\pdensities$ such that the ordinary products $GF_1$ and $GF_2$ are vector fields in $H\pdensities$. As
\begin{equation*}
  D_G \expectat p {F_1(p) F_2(p)} = \expectat p {D_G F_1(p) F_2(p)} + \expectat p {F_1(p) D_G        F_2(p)},
\end{equation*}
$D_G$ is a metric derivative.
\end{enumerate}
\end{proposition}

\section{Deformed exponential manifold}
\label{cha:deform-expon-manif}
The deformed exponential function is defined in \cite[Ch. 10]{naudts:2011GTh} as the inverse function of a deformed logarithm, with the aim to define a generalisation of entropy and exponential families. To improve consistency with the literature the $\phi$-notation in this section differs from what was used in previous sections. 

Assume the function $\phi \colon \reals_> \to ]0,\phi(\infty)[$ is surjective, increasing and continuous.  The $\phi$-\emph{logarithm} is the function
\begin{equation}\label{eq:philn}
  \philnof v = \int_1^v \frac{dx}{\phi(x)}, \quad v \in \reals_>.
\end{equation}
The $\phi$-logarithm, also called \emph{deformed logarithm}, $\Philn$ is defined on $\reals_>$ and it is strictly increasing, concave and differentiable. Its values range between $-\int_{0}^1 \frac{dx}{\phi(x)}$ and $\int_1^{+\infty} \frac{dx}{\phi(x)}$. If $\int_1^{+\infty} \frac{dx}{\phi(x)} = +\infty$, the range is of $\Philn$ is $]-m,+\infty[$ with $m = \int_{0}^1 \frac{dx}{\phi(x)} > 0$.
We assume the $\phi$ function is affinely bounded, so that
\begin{equation*}
\lim_{u\to\infty} \philnof u \ge \int_1^{+\infty} \frac{dx}{Ax+B} = +\infty.  
\end{equation*}

The $\phi$-\emph{exponential} or \emph{deformed exponential} is the inverse function of $\Philn$,
\begin{equation*}
  \Phiexp = \Philn^{-1} \colon ]-m,+\infty[ \to \reals_>.
\end{equation*}
It is positive, increasing, convex, differentiable.

\begin{example}[Tsallis logarithm and exponential \cite{tsallis:1988}]
The Tsallis logarithm with parameter $q \in ]0,1]$ is a deformed logarithm with $\phi(v)=1/v^q$. We have the explicit form
\begin{equation*}
  \qlnof v = \int_1^{v} \frac{dx}{x^q} =
  \begin{cases}
    \frac1{1-q}\left(v^{1-q}-1\right), & q\in ]0, 1[, \\ \lnof v, & q=1.
  \end{cases}
\end{equation*}

The corresponding exponential is defined for $q\ne 1$ by
\begin{equation*}
  \qexpof u = \left(1+(1-q)u\right)^{\frac1{1-q}}, \quad u > -\frac1{1-q} = m. 
\end{equation*}
\end{example}

\begin{example}[Kaniadakis exponential and logarithm \cite{kaniadakis:2002PhRE,kaniadakis:2005PhRE}]
The Kaniadakis exponential with parameter $\kappa \in [0,1[$ is based on the function
\begin{equation*}
    \phi(x) = \frac{2x}{x^\kappa+x^{-\kappa}} = \frac{2x^{\kappa+1}}{x^{2\kappa}+1},\quad x > 0.
\end{equation*}
This function is linearly bounded, $\phi(x) \le x$; it is equivalent to $2x$ for $x \downarrow 0$ and to $2x^{1-\kappa}$ for $x \uparrow +\infty$. 

The deformed logarithm is
\begin{equation*}
\klnof v = \int_1^v \frac{x^\kappa+x^{-\kappa}}{2}\frac1x dx =
\begin{cases}
\frac1{2\kappa}\left(v^\kappa-v^{-\kappa}\right) & \text{if $\kappa \neq 0$}, \\
\ln v & \text{if $\kappa = 0$}.
\end{cases}
\end{equation*}

By checking the differential equation $y'=\phi(y)$ one shows that the deformed exponential is 
\begin{equation*}
  \kexpof u = \expof{\int_0^u \frac{dy}{\sqrt{1+\kappa^2 y^2}}} = \begin{cases}
\left(\kappa u + \sqrt{1+\kappa^2 u^2}\right)^{\frac1\kappa} & \text{if $\kappa \neq 0$}, \\
\exp u & \text{if $\kappa = 0$}.
\end{cases}
\end{equation*}
\end{example}
\begin{example}[Nigel J. Newton exponential \cite{newton:2012}]
The function
\begin{equation*}
  \phi(x) = \frac{x}{x+1}, \quad x > 0,
\end{equation*}
has image $\rangeof \phi = ]0,1[$, is bounded by 1 and is linearly bounded by $x$. The $\phi$-logarithm is
\begin{equation*}
  \philnof u = \int_1^{u} \frac{x+1}{x} dx = u -1 + \ln u. 
\end{equation*}
\end{example}
\subsection{Model space}
Here we built our model spaces according to the proposal of Vigelis and Cavalcante \cite{vigelis|cavalcante:2011}.  if $\phi(x)=x$. Note that the $\phi$-exponential notation is not used in \cite{vigelis|cavalcante:2011}, where $\phi$ denotes a class of deformed exponential function larger than the one used here.
\begin{definition}
For each $p \in \pdensities$, we define the vector space
\begin{equation*}
  L^{\phi,p}(\mu) = \setof{u}{\exists \alpha > 0  \int \phiexpof{\alpha \absoluteval u + \philnof p} d\mu < + \infty}
\end{equation*}
\end{definition}
The vector space property is a consequence of the convexity of the $\phi$-exponential. Under our assumptions on $\mu$ (locally finite) and on $\phi$ (affinely bounded) such vector spaces are not empty. Bounded random variables whose support has finite $\mu$ measure belong to each $L^{\phi,p}(\mu)$. 
\begin{proposition}
The following statements are equivalent to $u \in L^{\phi,p}(\mu)$.
 \begin{enumerate}
\item For all real $\theta$ in a neighborhood of 0
\begin{equation*}
    \int \phiexpof{\theta u + \philnof p} d\mu < + \infty.
  \end{equation*}
 \item For some positive $\alpha$ 
\begin{equation*}
    \frac12 \left(\int \phiexpof{\alpha u + \philnof p} d\mu + \int \phiexpof{\alpha u + \philnof p} d\mu \right) < + \infty.
  \end{equation*}
\end{enumerate}
\end{proposition}
For each $u \in L^{\phi,p}(\mu)$ the set
\begin{equation*}
  \setof{r > 0}{\int \phiexpof{r^{-1}\absoluteval u + \philnof p} d\mu \le 2}
\end{equation*}
is an infinite interval of the positive real line. Its left end is the norm of $u$.
\begin{proposition}
  The vector space $ L^{\phi,p}(\mu) $ is a Banach space for the norm
  \begin{equation*}
    \Vert u \Vert_{\phi,p} = \inf \setof{r > 0}{\int \phiexpof{r^{-1}\absoluteval u + \philnof p} d\mu \le 2}.
  \end{equation*}
\end{proposition}
The importance of escort measures in deformed exponential families have been pointed out in \cite[\S 10.5]{naudts:2011GTh}. 

\begin{definition}\ 
\begin{enumerate}
\item The measure $\phi(p)\cdot\mu$ is equivalent to $\mu$ and it is
  called \emph{escort measure} of $p$.
\item If $\mu$ is a finite measure, or if $\phi$ is linearly bounded, then the escort measure is finite. In such a case, its formalized density is called the \emph{escort density} of $p$. We write 
\begin{equation*}
\expectat{\phi,p}{u} = \frac{\int u \phi(p) d\mu}{\int \phi(p) d\mu}.
\end{equation*}
\end{enumerate}
\end{definition}
\begin{proposition}\ 
\begin{enumerate}
\item The Banach space $L^{\phi,p}(\mu)$ is contained in the Lebesgue space $L^1(\phi(p)\cdot\mu)$ and the injection is non-expansive,
  \begin{equation*}
    \int \absoluteval u \phi(p) d\mu \le \Vert u \Vert_{\phi,p}.
  \end{equation*}
\item
$L^{\phi,p}(\mu)$ is a dense subspace of $L^{\phi,p}(\mu)$.
\item The space
\begin{equation*}
  T_{\phi,p} = \setof{u \in L^{\phi,p}(\mu)}{ \int u \phi(p) d\mu = 0}
\end{equation*}
is a closed subspace of $L^{\phi,p}(\mu) $ hence a Banach space for
the induced norm. 
\end{enumerate}
\end{proposition}
\begin{proof}
  From the convexity of $\Phiexp$
  \begin{equation*}
    \phiexpof{r^{-1} \absoluteval v + \philnof p} \ge p + \phi(p) r^{-1} \absoluteval v,
  \end{equation*}
hence,  if $r > \Vert v \Vert_{\phi,p}$,
\begin{equation*}
  2 \ge \int p d\mu + r^{-1} \int \absoluteval v \phi(p) d\mu.
\end{equation*}
It follows that $v \in L^1(\phi(p) \cdot\mu)$ and $\int u \phi(p) d\mu$ is well defined. Moreover,
\begin{equation*}
  \absoluteval{\int \absoluteval u \phi(p) d\mu} \le r
\end{equation*}
for all $r > \Vert v \Vert_{\phi,p}$.
\end{proof}
\begin{proposition}
  \begin{equation*}
    L^{\phi,q} \subset L^{\phi,p} \quad \Longleftrightarrow \quad \left(\philnof q - \philnof p\right) \in L^{\phi,p}(\mu)
  \end{equation*}
\end{proposition}
Given $p, q \in \pdensities$, convexity implies for $t \in [0,1]$
\begin{align*}
\int\phiexpof{t(\philnof q - \philnof p) + \philnof p} d\mu &= \\
  \int\phiexpof{(1-t)\philnof p + t\philnof q}d\mu &\le 1 < +\infty.
\end{align*}
The inequality shows that any two distinct positive probability densities are connected by a closed arc of of densities with total mass strictly smaller that one. The arc extends to negative values $t \in ]a,1] \supset [0,1]$ with densities of finite mass if, and only if, $\philnof q - \philnof p \in L^{\phi,p}(\mu)$. This leads to the following definition.
\begin{definition}\label{def:phiconnection}
 The densities $p, q \in \pdensities$ are \emph{$\phi$-connected by an open arc} if $]a,b[ \supset [0,1]$ and
  \begin{equation}\label{eq:oconnection}
      \int\phiexpof{(1-t)\philnof p + t\philnof q}d\mu < +\infty, \quad t \in ]a,b[.
  \end{equation}
\end{definition}
\begin{proposition}
The relation in Definition \ref{def:phiconnection} is an equivalence relation.
\end{proposition}
\begin{proof}
We show transitivity for densities $p,q,r \in posdensities$ with $p,q$ $\phi$-connected on $]a,b[$ and $q,r$ on $]c,d[$. The convex function
\begin{equation*}
  (\alpha,\beta) \mapsto \int\phiexpof{\alpha\philnof p + (1-\alpha-\beta)\philnof q + \beta\philnof r} d\mu
\end{equation*}
is finite on the interior of the convex hull of $(1-b,0),(1-a,0),(0,c),(0,d)$, hence $p,r$ are $\phi$-connected on $]-bd/(b+d-1),b(1-c)/(b-c)[$.
\end{proof}
\begin{proposition}\label{prop:phiconnection}
$L^{\phi,q}(\mu) = L^{\phi,p}(\mu)$ if, and only if, $p$ and $q$ are $\phi$-connected by an open arc.
\end{proposition}
\begin{proof} 
  Follows from the symmetric inclusion.
\end{proof}
If $p_0,p_1$ are $\phi$-connected by an open interval we can define a statistical model by
\begin{equation*}
  p(t) = \frac{\phiexpof{(1-t)\philnof{p_0}+t\philnof{p_1}}}{\int\phiexpof{(1-t)\philnof{p_0}+t\philnof{p_1}}d\mu}.
\end{equation*}
If we define
\begin{equation*}
  u = \philnof{p_1} - \philnof{p_0} - \int \left(\philnof{p_1} - \philnof{p_0}\right) \phi(p) d\mu, 
\end{equation*}
we can consider the expression for the density
\begin{equation*}
  \phiexpof{tu - \psi(t) + \philnof{p_0}} = p(t),
\end{equation*}
that is
\begin{equation*}
  tu - \psi(t) + \philnof{p_0} = \philnof{p(t)},
\end{equation*}
which gives
\begin{equation*}
  \psi(t) = \int \phi(p_0)\left(\philnof{p(t)} - \philnof{p_0}\right) d\mu.
\end{equation*}

\subsection{Generating functionals}
\label{sec:gener-funct}
\begin{definition}
  The convex function
  \begin{equation*}
    L^{\phi,p}(\mu) \ni v \mapsto \int \phiexpof{v + \philnof p} d\mu \in [1,+\infty]
  \end{equation*}
 is the \emph{$\phi$-moment generating functional at $p$}. 
\end{definition}
\begin{proposition}\ 
  \begin{enumerate}
  \item The proper domain of the moment generating functional contains the open ball of $L^{\phi,p}(\mu)$. The interior $\mathcal S_p$ of the proper domain of the moment generating functional in a nonempty convex open set.
  \item For each $u \in \mathcal S_p$ the mapping
    \begin{equation*}
      L^{\phi,p}(\mu) \ni v \mapsto \int \phi(\phiexpof{u+\philnof p}) v d\mu
    \end{equation*}
is linear and continuous.
\item The moment generating functional is lower semicontinuous.
  \end{enumerate}
\end{proposition}
\begin{proposition}\ 
\begin{enumerate}
\item For each $u \in B_p$ such that
\begin{equation}\label{eq:domainK}
  \int \phiexpof{u + \philnof p} d\mu < +\infty
\end{equation}
there exists a unique nonnegative constant $K_p(u)$, positive if $u \ne 0$, such that 
\begin{equation*}
 q = \phiexpof{u - K_p(u) + \philnof p}
\end{equation*}
is a probability density. 
\item In particular, \eqref{eq:domainK} holds if $\normat{\phi,p} u < 1$.
\item If $q$ is a positive density such that $\philnof p - \philnof q \in L^{\phi,p}(p)$, then
  \begin{equation}
    \label{eq:pchart}
    u = \philnof q - \philnof p - \int \phi(p)\left(\philnof q - \philnof p\right) d\mu
  \end{equation}
satifies \eqref{eq:domainK} and
\begin{equation}
  \label{eq:KisD}
  K_p(u) = \expectat {\phi(p)}{\philnof p - \philnof q}.
\end{equation}
\end{enumerate}
\end{proposition}
\begin{definition}
 The $\phi$-cumulant generating function is the funtion $K_p \colon T_p \to [0,+\infty]$ defined by
 \begin{equation*}
   K_p(u) = \sup\setof{k \ge 0}{ \int \phiexpof{u - k + \philnof p} d\mu \le 1}.
 \end{equation*}
\end{definition}

\begin{proposition}\ 
  \begin{enumerate}
  \item $K_p$ is null at 0, extended nonnegative and finite on the set
    \begin{equation*}
      \setof{u \in T_p}{\int \phiexpof{-u + \philnof p} d\mu < +\infty}.
    \end{equation*}
  \item If $K_p(u_0) < +\infty$, then for all $u \in T_p$
    \begin{equation*}
      K(u) \ge K(u_0) + \expectat{\Phi(p)}{u - u_0}.
    \end{equation*}
Therefore, $K_p$ is strictly convex, proper, lower semicontinuous.
  \item If $\philnof q - \philnof p \in \L^{\phi,p}(\mu)$ then its $\phi(p)$-centered random variable $u$ is in the proper domain of $K_p$ and viceversa.
\item The interior $\mathcal S$ of the proper domain of $K_p$ is a convex open set that contains the unit open ball $\setof{u \in L^{\phi,p|(\mu)}}{\normat{\phi,p} u < 1}$. On this set $K_p$ is differentiable and the derivative of $K_p$ at $u$ in the direction $v$ is 
  \begin{equation*}
    D K_p(u) v = \expectat{\phi(q)} u.
  \end{equation*}
  \end{enumerate}
\end{proposition}

At this point, we consider that almost all elements for the construction of a deformed exponential manifold along the same lines are available.

\section{Final remarks}
In this paper we have reviewed a specific track to the development of Information Geometry, i.e. the construction of a classical Banach manifold structure. This is done by developing in the natural way the original suggestion by B. Efron to look at the larger exponential structure. A non parametric approach is justified by the importance of applications essentially non parametric and by the neat mathematics involved. Other options are present in the literature, the most classical and most successful is based on the embedding $p \mapsto  \sqrt p$ from the probability density simplex into the $L^2$. Variants of this basic Hilbert embedding were used, see e.g. \cite{burdet|combe|nencka:2001}. S. Eguchi \cite{eguchi:2005igaia} has $L_0^2$ representation based on the mapping $u \mapsto \frac 12 - \frac 12 \sigma^2(u) + \frac 12 (1 - u)^2 = g$ which is defined on the unit $L_0^2$ open ball and takes its values in the set of densities which are bounded below by a positive constant. The duality between the exponential and mixture manifold, could lead to an other intermediate option, i.e. to define a manifold were the regularity of the maps is defined in some weak sense. See also the discussion in \cite{zhang|hasto:2006}. Another option uses a non-exponential representation of positive densities through the so-called deformed exponentials. 

\bibliographystyle{splncs}


\end{document}